\documentclass[a4paper, 11pt, article, oneside]{memoir}
\usepackage{CTT}
\newcommand\diffinternal{(1)\xspace}
\newcommand\diffuntyped{(2)\xspace}
\newcommand\diffrecursive{(3)\xspace}
\bibliography{CTT.bib}
\title{Against Cumulative Type Theory}
\author{Tim Button and Robert Trueman}
\emailaddress{tim.button@ucl.ac.uk and rob.trueman@york.ac.uk}
\date{}
\begin{document}\midsloppy
\pagestyle{nicelypage}
\maketitletop
\selectbodyfont

\noindent\textcolor{blue}{This is a pre-print; the paper is forthcoming at \emph{Review of Symbolic Logic}.}

\begin{abstract}
    Standard Type Theory, \STT, tells us that $b^n(a^m)$ is well-formed iff $n=m+1$. However, \textcite{LinneboRayo:HOI} have advocated for the use of Cumulative Type Theory, \CTT, which has more relaxed type-restrictions: according to \CTT, $b^\beta(a^\alpha)$ is well-formed iff $\beta>\alpha$. In this paper, we set ourselves against \CTT. We begin our case by arguing against \citeauthor{LinneboRayo:HOI}'s claim that \CTT sheds new philosophical light on set theory. We then argue that, while \CTT's type-restrictions are unjustifiable, the type-restrictions imposed by \STT are justified by a Fregean semantics. What is more, this Fregean semantics provides us with a principled way to resist \citeauthor{LinneboRayo:HOI}'s Semantic Argument for \CTT. We end by examining an alternative approach to cumulative types due to \textcite{FlorioJones:UQSTT}; we argue that their theory is best seen as a misleadingly formulated version of \STT.
\end{abstract}

\noindent 
Standard Type Theory, \STT, tells us that $b^n(a^m)$ is well-formed iff $n=m+1$. However, \textcite{LinneboRayo:HOI} have advocated for the use of Cumulative Type Theory, \CTT, which has more relaxed type-restrictions: according to \CTT, $b^\beta(a^\alpha)$ is well-formed iff $\beta>\alpha$. Other philosophers, including \textcites[237--8]{Williamson:MLM}[527]{Kramer:ETS}{FlorioJones:UQSTT}, have since expressed sympathy for cumulative types.

We set ourselves against cumulative type theory. We begin our case by arguing against \citeauthor{LinneboRayo:HOI}'s claim that \CTT sheds new philosophical light on set theory: in \S\ref{s:LR:interpreting} we highlight some important mathematical differences between \CTT and set theory, and in \S\ref{s:DiscussingSetsFromTypesTheorem} we explore the philosophical consequences of these differences. Then, in \S\ref{CTT-unjustified}, we push our case against \CTT further, by arguing that the type-restrictions it imposes are unjustifiable. This marks an important difference between \CTT and \STT: a \emph{Fregean} semantics justifies \STT's type-restrictions (see \S\ref{STT-justified}), and this Fregean semantics also provides us with a principled way to resist \citeauthor{LinneboRayo:HOI}'s Semantic Argument for \CTT (see \S\ref{optimism and union}). We end, in \S\ref{type theories:CTTfj}, by examining an alternative approach to cumulative types due to \textcite{FlorioJones:UQSTT}; we argue that their theory is best seen as a misleadingly formulated version of \STT.

\section{Formal type theories}
\label{type theories}

We start by outlining the formalisms of \STT and \CTT. For simplicity of exposition, in this paper we focus on \emph{monadic} type theories. (We also only consider un-ramified type theories.)

\subsection{\STT}
\label{type theories:STT}

\STT has a countable infinity of types, $0 \leq n < \omega$. The type of a term is indicated with a numerical superscript: $a^n$ is a type $n$ term. We have constants and variables of every type. Atomic formulas are made by combining a type $n\mathord{+}1$ term with a type $n$ term: $b^n(a^m)$ is well-formed iff $n = m + 1$. Intuitively, $b^{n+1}(a^n)$ applies a type $n\mathord{+}1$ entity to a type $n$ entity, where an entity is of type $n$ iff it is a value of a type $n$ variable; however, exactly what this intuitive gloss amounts to will depend on your preferred interpretation of the types (see \S\S\ref{CTT-unjustified}--\ref{STT-justified}).

Every type of variable can be bound by quantifiers. We here present the rules for $\forall$; the rules for $\exists$ are the obvious duals. For all types $n$, the following inferences are licensed, provided that (i) all expressions are well-formed, and (ii) $b^n$ does not occur in any undischarged assumptions on which $\phi(b^n)$ depends:
\begin{multicols}{2}
\begin{naturaldeduction}
    \AxiomC{$\forall x^n \phi(x^n)$}\inferencerule{\AllE{n}{}}
    \UnaryInfC{$\phantom{\forall x^n}\phi(a^n)$}
\end{naturaldeduction}
\begin{naturaldeduction}
    \AxiomC{$\phantom{\forall x^n}\phi(b^n)$}\inferencerule{\AllI{n}{}}
    \UnaryInfC{$\forall x^n \phi(x^n)$}
\end{naturaldeduction}
\end{multicols}
To ensure that each level of the type hierarchy is well-populated, we have the following scheme, for each type $n$:
\begin{listclean}
    \item[\emph{\STT-Comprehension.}] $\exists z^{n+1}\forall x^n(z^{n+1}(x^n)\liff \phi(x^n))$, whenever $\phi(x^n)$ is well-formed and does not contain $z^{n+1}$.
\end{listclean}
\STT has the usual stock of logical devices: quantifiers, connectives, and the identity sign, $=$. The identity sign can be flanked by a pair of terms of any type, but they must be terms of the same type; so $a^m=b^n$ is well-formed iff $m=n$. Identity is governed by the following scheme, for each type $n$:
    $$x^n=y^n \text{ iff } \forall z^{n+1}(z^{n+1}(x^n)\liff z^{n+1}(y^n))$$
We can treat this as an axiom scheme or an explicit definition. But, either way, $x^n=y^n$ is \emph{typically ambiguous}: there is not a single identity relation that applies across all the types, but a different relation for each type.

\subsection{\CTT}
\label{type theories:CTTlr}
\textcite{LinneboRayo:HOI} ask us to consider an alternative, cumulative, type theory, \CTT. This type theory was formally developed by \textcite{DegenJohannsen:CHOL}. (We discuss a different approach to cumulation, due to \textcite{FlorioJones:UQSTT}, in \S\ref{type theories:CTTfj}.) The basic thought behind \CTT is that the entities cumulate as you ascend through the types. Let us see how this is implemented.

First, \CTT relaxes \STT's syntax. In \STT, $b^n(a^m)$ is well-formed iff $n = m +1$. But, if the types cumulate, then everything at level $0$ reappears at level $1$; so, since $c^2(a^1)$ is meaningful, $c^2(a^0)$ should be too. More generally, \CTT allows that $b^\beta(a^\alpha)$ is well-formed iff $\beta > \alpha$. And note that we use `$\alpha$' and `$\beta$' rather than `$n$' and `$m$' here: if the types cumulate, we will want to be able to consider transfinite types, and so we must allow ourselves a transfinite stock of type-indices. (One obvious way to do this is to stipulate that the type-indices are von Neumann's ordinals, but the only important constraint is that the type-indices be well-ordered.)\footnote{Cf.\ \textcites[294]{LinneboRayo:HOI}[178--9]{LinneboRayo:RFS} on `definite' collections of languages and alternative `labels'. For readability, we use standard ordinal notation in this paper, but this is easily eliminable; e.g.\ `$\alpha+1$' can be parsed as `the next index after $\alpha$', and `$\omega$' as `the first limit index'.}

Second, \CTT has rather permissive inference rules for quantifiers. (Again, we only outline the rules for $\forall$.) For all types $\beta \geq \alpha$, the following inferences are licensed, provided that (i) all expressions are well-formed, and (ii) $b^\beta$ does not occur in any undischarged assumption on which $\phi(b^\beta)$ depends:\footnote{These are the obvious natural-deduction versions of \citepossess[149]{DegenJohannsen:CHOL} sequent-calculus rules. \textcite[288]{LinneboRayo:HOI} are not specific on the rules they adopt, but  \parencite*[282n20]{LinneboRayo:HOI} appeal to a result from \textcite[150]{DegenJohannsen:CHOL} which uses these rules.}
\begin{multicols}{2}
\begin{naturaldeduction}
    \AxiomC{$\forall x^\beta \phi(x^\beta)$}\inferencerule{\AllE{\beta}{\alpha}}
    \UnaryInfC{$\phantom{\forall x^\beta}\phi(a^\alpha)$}
\end{naturaldeduction}
\begin{naturaldeduction}
    \AxiomC{$\phantom{\forall x^\beta}\phi(b^\beta)$}\inferencerule{\AllI{\beta}{\alpha}}
    \UnaryInfC{$\forall x^\alpha \phi(x^\alpha)$}
\end{naturaldeduction}
\end{multicols}\noindent
These rules are intuitively sound, given the idea of cumulation: every type $\alpha$ entity is a type $\beta \geq \alpha$ entity too; so if $\phi$ holds of every type $\beta$ entity, then $\phi$ holds of each type $\alpha$ entity.

Third, to ensure that each successor-level of the type hierarchy is well-populated, \CTT has a {Comprehension} scheme, for each type $\alpha$:\footnote{\textcites[149]{DegenJohannsen:CHOL}[288]{LinneboRayo:HOI} offer a variant formulation, using $\lambda$-abstraction.}
\begin{listclean}
    \item[\emph{\CTT-Comprehension.}] $\exists z^{\alpha+1}\forall x^\alpha(z^{\alpha+1}(x^\alpha)\liff \phi(x^\alpha))$, whenever $\phi(x^\alpha)$ is well-formed and does not contain $z^{\alpha+1}$.
\end{listclean}
Fourth, \CTT has an infinitary inference rule for each limit type $\lambda$:\footnote{\textcites[153]{DegenJohannsen:CHOL}[288]{LinneboRayo:HOI}.}
\begin{center}
    \begin{naturaldeduction}
        \AxiomC{$\forall x^\alpha \phi(x^\alpha)$, for all $\alpha < \lambda$}\inferencerule{\emph{Limit$^\lambda$}}
        \UnaryInfC{$\forall x^\lambda\phi(x^\lambda)$\phantom{, for all $\alpha \leq \lambda$}}
    \end{naturaldeduction}
\end{center}
Intuitively, this guarantees that nothing essentially `new' happens at limit types, so that any type $\lambda$ entity is an entity of some type $\alpha < \lambda$. 

So far, we have identified entities across types quite freely. However, \textcite[281--3]{LinneboRayo:HOI} retain the rule that a strict identity claim, $x^\alpha=y^\beta$, is well-formed iff $\alpha=\beta$. To deal with cross-type identity, they explicitly define a new sign, $\eqCTT$, for any types $\alpha$ and $\beta$ and where $\gamma = \max(\alpha, \beta)+1$:\footnote{\textcite[149]{DegenJohannsen:CHOL} draw no  distinction between $=$ and $\eqCTT$.}
\begin{align*}
	a^\alpha \eqCTT b^\beta & \mliffdf \forall x^{\gamma}(x^{\gamma}(a^\alpha) \liff x^\gamma(b^\beta))
\end{align*}
This definition is typically ambiguous: it defines different relations for different $\alpha$ and $\beta$. But all of these relations behave like identity: if $\phi(a^\alpha)$ and $\phi(b^\beta)$ are both well-formed, then $\phi(a^\alpha)$ and $a^\alpha\equiv b^\beta$ together entail $\phi(b^\beta)$.\footnote{For a proof, see Lemma \ref{fact:CTT} of \S\ref{s:app:elementary}.} Now we can prove the following theorem scheme, for all $\alpha\leq\beta$:\footnote{\label{fn:TypeRaisingProof}\textcite[288]{LinneboRayo:HOI} take Type-Raising as an axiom scheme; we prove it in Lemma \ref{fact:TR} of \S\ref{s:app:elementary}.}
\begin{listclean}
    \item[\emph{Type-Raising Scheme.}] $\forall x^\alpha \exists y^\beta\phantom{(}x^\alpha\equiv y^\beta$
\end{listclean}
So, if $\alpha \leq \beta$, then every type $\alpha$ entity is a type $\beta$ entity, in the sense of `is' expressed by $\eqCTT$. 

We also provide another (typically ambiguous) explicit definition, where $\gamma = \max(\alpha, \beta)+1$:\footnote{\textcites[151]{DegenJohannsen:CHOL}[282]{LinneboRayo:HOI}. \emph{Notation}: $(\exists x^\gamma \eqCTT b^\beta)\phi$ abbreviates $\exists x^\gamma(x^\gamma \eqCTT b^\beta \land \phi)$; similarly, $(\forall x^\gamma \eqCTT b^\beta)\phi$ abbreviates $\forall x^\gamma(x^\gamma \eqCTT b^\beta \lonlyif \phi)$; and similarly for other two-place infix predicates.}
\begin{align*}
	a^\alpha \inCTT b^\beta & \mliffdf (\exists x^{\gamma} \eqCTT b^\beta)x^{\gamma}(a^\alpha)
\end{align*}
This membership-like notion applies $b^\beta$ to $a^\alpha$, but is well-formed for every $\alpha$ and $\beta$. So $a^\alpha \inCTT b^\beta$ allows us to simulate $b^\beta(a^\alpha)$, even when $\alpha\geq\beta$. 

If we provide no further axioms, though, then $\inCTT$ can be ill-founded. To rule this out, we lay down two final schemes, for all $\alpha, \beta$:\footnote{In \S\ref{s:app:elementary}, we prove that Type-Founded and Type-Base are independent from the axioms given so far. \textcite[289]{LinneboRayo:HOI} provide a version of \emph{Type-Base}, but no version of \emph{Type-Founded} (though they clearly want some such principle; see \cite[283n.22]{LinneboRayo:HOI}). \textcite{DegenJohannsen:CHOL} tackle this slightly differently; see the start of \S\ref{s:PCTT:appendix}, below. With these principles, we can establish that: if $\alpha < \beta$ then $a^\alpha \inCTT b^\beta$ iff $b^\beta(a^\alpha)$; if $\alpha \geq \beta$ and $\alpha$ is minimal for $a^\alpha$ and $\beta$ is minimal for $b^\beta$, then $a^\alpha \notinCTT b^\beta$. (Here, we say that $\gamma$ is minimal for $c^\gamma$ iff $\forall x^\delta\ c^\gamma \neqCTT x^\delta$ for all $\delta < \gamma$.)}
\begin{listclean}
    \item[\emph{Type-Founded}.] $\forall a^\alpha \forall b^{\beta+1}(a^\alpha \inCTT b^{\beta+1} \lonlyif \exists x^\beta\ a^\alpha \eqCTT x^\beta)$    
    \item[\emph{Type-Base}.] $\forall x^0 \forall y^\alpha\ y^\alpha \notinCTT x^0$    
\end{listclean}
This completes the list of axioms and inference rules for \CTT.

It is worth making a brief observation about syntax. In moving from \STT to \CTT, we are asked to relax \STT's syntax: $b^\beta(a^\alpha)$ is well-formed iff $\beta > \alpha$. There is an obvious way to relax this further, whilst retaining a typed theory: allow that $b^\beta(a^\alpha)$ is well-formed for \emph{any} $\alpha$ and $\beta$. However, this further relaxation would have no real effect. As just noted, \CTT can simulate $b^\beta(a^\alpha)$ using the formula $a^\alpha \inCTT b^\beta$, where the latter is defined using the more stringent type-restrictions. Consequently, we can be largely indifferent on whether to use the \emph{stringent} type-restrictions, so that $b^\beta(a^\alpha)$ is well-formed iff $\beta > \alpha$, or the more \emph{liberal} type-restrictions, so that $b^\beta(a^\alpha)$ is well-formed for any $\alpha$ and $\beta$. In what follows, we will tend to adopt the stringent type-restrictions, but we will revisit this in \S\ref{CTT-unjustified}.

For each type-index $\tau$, the theory \CTT[\tau] has a countable infinity of distinct variables of every type $< \tau$, and no terms of any type $\geq \tau$. We refer to the cumulative type theories in general as `\CTT', using `\CTT[\tau]' with the superscript when it is important to pay attention to the bound.

\section{The Sets-from-Types Theorem}
\label{s:LR:interpreting}

\textcites{DegenJohannsen:CHOL}{LinneboRayo:HOI} note that \CTT interprets an iterative set theory. In this section, we present a strengthened version of their formal results. We discuss its philosophical significance in \S\ref{s:DiscussingSetsFromTypesTheorem}. For ease of exposition, we will consider set theories without urelements (and similar type theories); we could accommodate urelements if we liked, but it would complicate our discussion without adding any real insight.

\subsection{The interpretation}\label{s:LR:theinterpretation}
We will focus on a `pure' version of \CTT[\tau], which we call \PCTT[\tau]. This augments \CTT[\tau] with principles guaranteeing that there is exactly one type $0$ entity, and that coextensive entities at higher-types are identical. (For details, see \S\ref{s:formulatingPCTT}.) The set theory that \PCTT[\tau] can interpret is \Zr, i.e.\ \emph{Z}ermelo set theory together with the principle that the sets are arranged into well-ordered \emph{r}anks. This theory omits Replacement, and so is strictly weaker than \ZF. (For more details, see \S\ref{s:PCTT:appendix}; note that \ZF = \Zr + Replacement.) 

To interpret \Zr with \PCTT[\tau], we first define a translation. For each \Zr-formula $\phi$, let \typetrans{$\phi$}{\kappa} be the formula which results by replacing each `$\in$' with `$\inCTT$', each `$=$' with `$\eqCTT$', and superscripting each variable with $\kappa$. For example, the Axiom of \typetrans{Powersets}{\kappa} is:
    $$\forall a^{\kappa}\exists b^{\kappa}\forall x^{\kappa}(x^{\kappa} \inCTT b^{\kappa} \liff (\forall v^\kappa \inCTT x^{\kappa})v^\kappa \inCTT a^{\kappa})$$
Now we can prove the following result (see \S\ref{s:PCTT:appendix}, Theorem \ref{thm:PCTT:Zr}):\footnote{\label{fn:LR:givethemranks}\label{fn:kappa+2explanation}This extends \citepossess[\S4.1]{DegenJohannsen:CHOL} results concerning \Zermelo. \textcite[289, n.28]{LinneboRayo:HOI} cover only \Zermelo without Foundation. The bound $\kappa + 2 < \tau$ is needed as $a^\kappa \inCTT b^\kappa$ abbreviates $\exists x^{\kappa+1}(\forall z^{\kappa+2}(z^{\kappa+2}(x^{\kappa+1})\liff z^{\kappa+2}(b^{\kappa})) \land x^{\kappa+1}(a^\kappa))$.}
\begin{listclean}
    \item[\emph{The Sets-from-Types Theorem.}]
    $\PCTT[\tau] \proves \typetrans{Zr}{\kappa}$, for any limit $\kappa > \omega$ with $\kappa + 2 < \tau$.
\end{listclean}\noindent
Otherwise put: \PCTT[\tau] proves the translations of all theorems of \Zr.

\subsection{Differences between \Zr and \Zrkappa}\label{s:LR:interpret:differences}
The proof of the Sets-from-Types Theorem involves establishing a tight association between two notions: an entity's \emph{type}, as in \PCTT[\tau] and \typetrans{\Zr}{\kappa}, and a set's \emph{rank}, as in \Zr. This sort of connection leads \textcite[289]{LinneboRayo:HOI} to claim that `there is no deep \emph{mathematical} difference between the ideological hierarchy of type theory and the ontological hierarchy of set theory.' 

Whether to describe them as `deep' may be a matter of taste, but it is worth noting three mathematical differences between \Zr, on the one hand, and \PCTT[\tau] and \typetrans{\Zr}{\kappa}, on the other.\footnote{\textcite[284, 289]{LinneboRayo:HOI} mention differences \diffinternal and 
\diffuntyped themselves, but they do not mention \diffrecursive.} We summarize the differences in the following table:
\begin{center}

\

	\noindent\begin{tabular}{@{}p{0.05\textwidth} p{0.35\textwidth} p{0.35\textwidth}@{}}
		\toprule
		& \Zr & \PCTT[\tau] and \typetrans{\Zr}{\kappa} 
		\tabularnewline\addlinespace
		\midrule \addlinespace
        \diffinternal & ranks defined internally & types supplied externally \\
        \diffuntyped & untyped variables & typed variables \\
        \diffrecursive & recursively axiomatized & arithmetically complete 
        \tabularnewline\addlinespace
	    \bottomrule
	\end{tabular}
	
\

\end{center}
We will now explain these three differences.

\emph{Concerning \diffinternal.} The notion of 
\emph{rank} is explicitly defined within \Zr itself, much as it is within \ZF.\footnote{Indeed, it is definable within \LT; see \S\ref{s:PCTT:appendix}.} By contrast, the notion of \emph{type} is metatheoretic for both \PCTT[\tau] and \typetrans{\Zr}{\kappa}. Every variable carries a type-index, and these type-indices are supplied \emph{externally}. Indeed, when we take the very first step of describing the syntax of \CTT[\tau], we assume as given all the {type-indices} $< \tau$.

\emph{Concerning \diffuntyped.} \Zr is essentially \emph{untyped}. It has exactly one kind of variable, which ranges over all sets of all ranks. By contrast, every variable in \PCTT[\tau] and \typetrans{\Zr}{\kappa} carries a type-index, and \PCTT[\tau]'s quantifier rules indicate that type $\alpha$ variables range only over entities of type $\leq \alpha$. These theories have no \emph{untyped} variables; that is, they have no variables which range over all entities of \emph{all} types. (Note that, despite our use of the phrase `ranging over', this difference shows up at the level of the formal theories, prior to interpretation. Indeed, none of the differences depend upon any semantic considerations.)

\emph{Concerning \diffrecursive.} Clearly, \Zr is recursively axiomatized (see \S\ref{s:formulatingZr}). However, neither \PCTT[\tau] nor \typetrans{\Zr}{\kappa} is recursively axiomatizable, thanks to the intrinsically infinitary Limit-rules. Indeed, Limit$^\omega$ makes these theories arithmetically complete, since it simulates Hilbert's $\omega$-rule.\footnote{Assuming $\kappa > \omega$ and $\kappa + 2 < \tau$. \emph{Sketch.} Using the Sets-from-Types Theorem, use \PCTT[\tau] to develop \typetrans{Zr}{\omega+\omega}. In \typetrans{Zr}{\omega+\omega}, define $\mathbb{N}$ as the set of finite von Neumann ordinals, and define $+$ and $\times$ as usual. Suppose we can show $\phi(n)$ for each $n$; then since the type of each $n$ is $n$, for each $n < \omega$ we can show $\forall x^n(x^n \inCTT \mathbb{N} \lonlyif \phi(x^n))$; now use Limit$^\omega$.}

\subsection{Mathematical foundations}\label{s:LR:interpret:foundations}
We will now explain why these three differences are mathematically significant. In brief: the differences show that \Zr is expressively richer but deductively weaker than \Zrkappa; this makes \Zr much more suitable as a framework for considering  mathematical foundations. 

Differences \diffinternal--\diffuntyped show that \Zr is \emph{expressively richer} than \Zrkappa. To see this, consider how we might formulate questions about the height of a hierarchy. In the case of \Zr, we might ask a specific question like: \emph{Should we countenance a strongly inaccessible rank?} That question is formulated within the \emph{object language} of \Zr, and this is possible because \Zr's untyped variables range over all the sets, whatever their rank. So, whilst \Zr does not \emph{settle} whether there are any sets of strongly inaccessible rank, it does allow us to \emph{formulate} the claim that there are, and tells us that any such sets obey Extensionality and Separation (for example). In the case of \PCTT, the analogous question about the height of a type-hierarchy would be: \emph{Should we countenance a strongly inaccessible type-index?} But this question is, of course, formulated within a \emph{metalanguage}. After all, each \PCTT[\tau] has variables of all and only the types $< \tau$, and \Zrkappa has variables of all and only the types $\leq \kappa+2 < \tau$,\footnote{See footnote \ref{fn:kappa+2explanation}.} so neither theory allows us to formulate questions about entities of type $\tau$; they literally lack the vocabulary for doing so. 

Difference \diffrecursive, however, shows that \Zr is \emph{deductively weaker} than \Zrkappa. This is obvious---one is arithmetically complete, the other is not---but let us draw out a couple of consequences. The Sets-from-Types Theorem tells us that \PCTT[\tau] interprets \Zr. However, this interpretation is not \emph{faithful}, i.e.\ some non-theorems of \Zr become theorems of \Zrkappa under interpretation; nor is the interpretation \emph{mutual}, i.e.\ \Zr cannot interpret \Zrkappa.\footnote{\emph{Illustration.} Let $\text{Con}_{\Zr}$ be a suitable consistency sentence for \Zr. This is independent from \Zr, by the second incompleteness theorem; but \Zrkappa proves $\text{Con}_{\Zr}^{(\kappa)}$, since it is arithmetically complete. The same example shows that \Zr does not interpret \Zrkappa.}

This combination of expressive richness with deductive weakness makes \Zr much more suitable as a framework for mathematical foundations than \Zrkappa or  \PCTT[\tau].\footnote{\emph{Pace} \citepossess{DegenJohannsen:CHOL} sentiment \PCTT might serve `as a foundation for set theory'. Note that differences \diffinternal--\diffuntyped also underpin the philosophical discussion of \S\ref{s:DiscussingSetsFromTypesTheorem}.} Concerning expressive strength: if our hierarchy is to serve as any kind of mathematical foundation, then questions about the height of the hierarchy will be of pressing importance; but only \Zr provides a suitable framework for raising such questions. Concerning deductive weakness: any adequate foundational theory must be recursively axiomatizable since, as  \textcite[45]{Godel:PSFM} put it, only recursively axiomatizable theories can leave no doubt regarding whether a putative proof \emph{is} a proof, so that `the highest possible degree of exactness is obtained'; but only \Zr is recursively axiomatized. 

\subsection{Gödel on `superfluous restrictions' in type theory}\label{s:GodelSuperfluous}
We just quoted Gödel on mathematical foundations. Having made the quoted remarks, Gödel went on to make a more famous claim:
\begin{quote}
	the system of axioms for the theory of aggregates, as presented by Zermelo, Fraenkel, and von Neumann\ldots is nothing else but a natural generalization of the theory of types, or rather, what becomes of the theory of types if certain superfluous restrictions are removed.\footnote{\textcite[45--6]{Godel:PSFM}.}
\end{quote}
He continued by outlining the `superfluous restrictions' thus:\footnote{\textcite[46--7]{Godel:PSFM}; for discussion, see \textcites[37]{Feferman:Note1933}[88]{Tait:GUPFM}.}
\begin{listr-0}
	\item\label{godel:relax:false} `$a \in b$' is meaningful iff the type of `$b$' is exactly one greater than that of `$a$';
	\item\label{godel:relax:class} each class (of any type) can contain classes of exactly one type;
	\item\label{godel:relax:limit} only finite types are allowed.
\end{listr-0}
Whilst explicitly disavowing exegetical aims, \textcite[273--4, 278]{LinneboRayo:HOI} motivate \CTT by suggesting that \CTT arises from \STT simply by lifting these `superfluous restrictions'. 

Certainly \CTT lifts restrictions \eqref{godel:relax:false}--\eqref{godel:relax:limit}. But \Zr also lifts these restrictions, and in a different way. Moreover, it is this latter way which we find in Gödel's \cite*{Godel:PSFM} lecture. On each of points \diffinternal--\diffrecursive from \S\ref{s:LR:interpret:differences}, Gödel sides against the use of anything like \Zrkappa.

\emph{Concerning \diffinternal.} \textcite[47]{Godel:PSFM} is clear that the theory which arises by removing \STT's `superfluous restrictions' will supply its own `types'.\footnote{\textcite[92]{Tait:GUPFM} emphasises this point, and \textcite[289 n.28]{LinneboRayo:HOI} concede it.}

\emph{Concerning \diffuntyped.} \textcite[49]{Godel:PSFM} complains that, in 
\STT, we have to formulate `the logical axioms for each type separately', and he states that the theory which removes \STT's `superfluous restrictions' will avoid this complaint. Such a theory will therefore employ an \emph{untyped} variable, which can range over all entities.

\emph{Concerning \diffrecursive}. As already noted, \textcite[45, 48]{Godel:PSFM} insists that an adequate formalization of the foundations of mathematics must be recursively axiomatizable, and explicitly remarks that such theories are necessarily  arithmetically incomplete.

Gödel, then, seems never to have envisaged theories like \Zrkappa or \PCTT[\tau].\footnote{Can we consider (or might Gödel have considered) the move from \STT to \Zr as involving two steps: \emph{first}, \citeauthor{LinneboRayo:HOI}'s step from \STT to some \CTT[\tau] and \Zrkappa; \emph{second} the addition of an untyped variable to \Zrkappa, yielding \Zr? (Thanks to an anonymous referee for posing this question.) This may be a useful heuristic, but it is slightly technically infelicituous, since the result of adding an untyped variable to \Zrkappa will be  arithmetically complete (cf.\ footnote \ref{fn:pedantry}).} Rather, 
Gödel's suggestion was that removing \STT's `superfluous restrictions' led to \ZFU, by the simple stipulation that the `type' of $x$ is $\alpha$ iff $x \in V_{\alpha+1} \setminus V_\alpha$, with these segments of the set hierarchy defined directly \emph{within} \ZFU in the (now) familiar fashion.\footnote{Cf.\ \textcite[208]{Scott:AST}: `the best way to regard Zermelo's theory is as a simplification and extension of Russell's [\STT]\ldots. The simplification was to make the types \emph{cumulative}.' Note that we are talking about \ZFU rather than ZrU. This is inevitable, since ZrU was not formulated until long after Gödel's lecture. However, Gödel supplied an additional argument in favour of Replacement; see footnote \ref{footnote:godelbootstrap}, below.} That is, Gödel simply identified a set's `type' with (what we now call) its \emph{rank}, and advocated for the use of recursively axiomatized theories whose untyped variables range over all the sets (of all ranks).

\section{The (in)significance of the Sets-from-Types Theorem}\label{s:DiscussingSetsFromTypesTheorem}
We have noted the important {mathematical} differences between \Zr and \Zrkappa. We will now show how these differences undermine the {philosophical} significance of the Sets-from-Types Theorem. In broad brush strokes: \citeauthor{LinneboRayo:HOI} think that the Sets-from-Types Theorem sheds important new light on set theory; we disagree, since \Zrkappa and \Zr and importantly distinct. 

\subsection{Elsa's worries}
\label{hierarchies}
To reconstruct \citepossess[289--94]{LinneboRayo:HOI}[178]{LinneboRayo:RFS} intended use of the Sets-from-Types Theorem, we will introduce a character, Elsa. Elsa wants to use \Zr to talk about the hierarchy of sets, but she has some ontological worries. Following post-Quinean orthodoxy, Elsa draws a sharp distinction between a theory's ontology and its ideology. In general, Elsa thinks that if a theory is coherent, then that is enough to guarantee the good standing of its ideology: roughly, Elsa thinks that a theory's ideology merely provides you with a way of talking about objects, and there is no standard beyond coherence by which to judge ways of talking. Now, Elsa is certain that \Zr is coherent, and so she has no reservations about its ideology. But, \Zr also postulates a rich ontology of sets, and Elsa insists that the mere coherence of a theory is not enough to guarantee the existence of its ontological commitments. So, Elsa worries: \emph{What guarantees that there are enough sets?}

\citeauthor{LinneboRayo:HOI} have a sequence of recommendations for Elsa. First, they will introduce Elsa to the type hierarchy, in the form of \PCTT[\tau], whose coherence can be assumed (at least, in this context). The question arises of how Elsa should think about ontology/ideology in the type-theoretic context. Quantification over type $0$ entities is just first-order quantification; so Elsa should think that theorizing at type $0$ introduces \emph{ontological} commitments. However, Elsa can perhaps be encouraged to think that theorizing at higher types simply gives us sophisticated ways to talk about the objects at type $0$, and so only introduces \emph{ideological} commitments. If Elsa agrees to think in this way, then she will map her dichotomy between ontology and ideology onto the dichotomy between \emph{type $0$} and \emph{type $> 0$}.\footnote{\textcite[270]{LinneboRayo:HOI} claim that this is how the higher-types are widely regarded by philosophers. For the record, we think that anyone who uses type theory (cumulative or non-cumulative) should reject the idea that there is a useful ontology/ideology dichotomy to be drawn along this faultline. When \textcite{Quine:OI} drew his distinction between ontology and ideology, he drew it for \emph{first-order logic}. In that setting, the distinction is clear enough: we are ontologically committed to the things we quantify over; ideological commitments are expressed by symbols in positions that cannot be quantified into. But in a type-theoretic setting, we can quantify into predicate-position. So distinctions of logical order no longer align with the quantifiable/unquantifiable distinction. See also  \textcites[260--1]{Williamson:MLM}[ch.7]{Trueman:PP}.} Having done this, she will regard \PCTT[\tau] as ontologically unproblematic: it posits just one object (i.e.\ one type $0$ entity). Granted, she may regard \PCTT[\tau] as ideologically profligate, but she thinks that its coherence guarantees the good standing of its ideology. Consequently, Elsa should have no worries about using \PCTT[\tau]. Now, via the Set-from-Types Theorem, Elsa can use \PCTT[\tau] to obtain \Zrkappa. So, according to \citeauthor{LinneboRayo:HOI}, Elsa will have no reason to worry about using \Zrkappa in place of \Zr.

Having come this far, \textcite[290]{LinneboRayo:HOI} hope that Elsa might now be brought to share their view, that `the two hierarchies'---the `ideological' hierarchy of \Zrkappa and the `ontological' hierarchy of \Zr---`constitute different perspectives on the same subject-matter.' But we do not need to consider that further step. We think that Elsa should balk at the line of reasoning given in the previous paragraph. 

\subsection{Ontology relocated}\label{s:LR:ontologyrelocated}
The immediate problem is that \Zr and \Zrkappa are importantly different theories. One of the differences, mentioned in  \S\ref{s:LR:interpret:foundations}, is that Elsa can ask about the height of her set-hierarchy within the object-language of \Zr, whereas she can only ask about the height of a type-hierarchy within a metalanguage. But, as we will now show, this basic issue---of object language versus metalanguage---completely undermines the dialectical force of \citeauthor{LinneboRayo:HOI}'s line of reasoning. 

Recall: Elsa wants to use \Zr, but worries: \emph{What guarantees that enough sets exist?} \citeauthor{LinneboRayo:HOI} recommend that Elsa invoke the Sets-from-Types Theorem. Specifically, they encourage Elsa to fix some limit $\kappa > \omega$ with $\kappa + 2 < \tau$, then work in \PCTT[\tau] to obtain \Zrkappa. 

Inevitably, though, this discussion of $\kappa$ and $\tau$ takes place within some \emph{metatheory} which we use to describe \PCTT[\tau]. After all, as noted in \S\S\ref{s:LR:interpret:differences}--\ref{s:LR:interpret:foundations}, \PCTT[\tau]'s types are supplied \emph{externally}. So, if Elsa is to follow \citeauthor{LinneboRayo:HOI}'s recommendation, she will have to countenance a suitably large index, $\tau$, in the metatheory, so that she can both describe \PCTT[\tau] and obtain \Zrkappa. 

At this point, though, Elsa will simply want to ask: \emph{What guarantees that any suitable $\tau$ exists?} Such an entity would have to stand at the head of a vast sequence of type-indices. Well then: \emph{What guarantees that enough type-indices exist?} Her ontological worries about sets have not have been {addressed}; they have just become worries about the ontology postulated within the {metatheory}. 

\subsection{Ideological-bootstrapping}\label{s:LR:bootstrapping}
This elementary problem undermines \citeauthor{LinneboRayo:HOI}'s way of dealing with Elsa. However, it is worth considering one possible line of response, via (what we call) \emph{ideological-bootstrapping}. This idea is independently interesting, and it will buy \citeauthor{LinneboRayo:HOI} some slack, but not enough slack to save their argumentative strategy.

To define \PCTT[\tau], we must be given the type-index $\tau$. In the previous subsection, we imagined Elsa worrying about whether $\tau$ exists. But---so this line of reply runs---Elsa is mistakenly assuming here that $\tau$ must be a type $0$ entity. Instead, $\tau$ could be a higher-type entity, supplied by some ideologically-rich but ontologically-innocent theory, \PCTT[\sigma]. In turn, $\sigma$ might be some higher-type entity, supplied by some theory \PCTT[\rho]. And so on.\footnote{\textcite[179]{LinneboRayo:RFS} suggest something a little similar, though in terms of the plurally-interpreted hierarchy (see \S\ref{s:illustrationplural}) and in response to a slightly different concern.}

The hope is that, somehow, considering a sequence of such theories will sooth away Elsa's ontological concerns. But, however exactly this line of response is meant to work, it will require that $\tau > \sigma > \rho > \ldots$. After all, Elsa's worries kick in as soon as the \emph{syntax} of \PCTT[\tau] is laid down; so her worries clearly cannot be addressed by starting with some theory \PCTT[\sigma] with $\sigma \geq \tau$. 

This simple observation dictates the form that the attempted reply must take. We are being asked to imagine a sequence of theories, \PCTT[\tau_1], \ldots, \PCTT[\tau_n], as follows:
\begin{listl-0}
    \item\label{bootstrap:base} $\tau_1$ is so small that Elsa has no serious qualms about its existence. 
    \item\label{bootstrap:induction} As we move along the sequence, the ideology strictly increases (i.e.\ $\tau_i < \tau_{i+1}$), but the earlier theory proves the existence of an entity which indexes the terms of the next theory (i.e.\ each \PCTT[\tau_i] proves the existence of something with order-type $\tau_{i+1} > \tau_i$).
    \item\label{bootstrap:payload} \PCTT[\tau_n] proves \Zrkappa, for some suitable $\kappa$.
\end{listl-0}
Call this response \emph{ideological-bootstrapping}, since ideologically weaker theories are used to define ideologically richer theories at step \eqref{bootstrap:induction}.\footnote{\label{footnote:godelbootstrap}\textcite[47]{Godel:PSFM} suggests something similar: given `the system $S_\alpha$ you can\ldots\ take an ordinal $\beta$ greater than $\alpha$ which can be defined in terms of the system $S_\alpha$, and by means of it state the axioms for the system $\beta$ including all types less than $\beta$, and so on.' However, Gödel is not discussing is not trying to defend anything like the argument of \S\ref{s:LR:ontologyrelocated}. As such---and unlike in the context of \emph{ideological}-bootstrapping---Gödel need not confine himself to finite sequences of theories. For discussion of Gödel, see \textcites[37--8n.h]{Feferman:Note1933}[89--93]{Tait:GUPFM}[21--24]{Koellner:SNA}[90--93]{Incurvati:CSFM}.}

(Note that we have assumed that the sequence of theories is \emph{finite}. To explain why, suppose someone instead suggests this: \emph{If Elsa has accepted the existence of an $\omega$-sequence of theories \PCTT[\tau_1], \PCTT[\tau_2], \ldots, then Elsa can bootstrap her way to their limit, \PCTT[\tau_\omega].} This suggestion is spurious. If some \PCTT[\tau_i] is sufficient to introduce an entity with order-type $\tau_\omega$, then we can simply take \PCTT[\tau_\omega] as the $i\mathord{+}1^\text{th}$ theory. The important case is when none of the theories \PCTT[\tau_i] suffices to introduce anything with order-type $\tau_\omega$. But in this case, Elsa will worry whether `taking the limit' is ontologically innocent; for, by assumption, she has not found any ontologically innocent theory which supplies $\tau_\omega$.)

Ideological-bootstrapping might work in specific circumstances. For example, suppose Elsa is comfortable with the existence of $\omega+\omega + 3$, and so has no concerns with the \emph{specification} of \PCTT[\omega+\omega+3]. Invoking the Sets-from-Types Theorem, \PCTT[\omega+\omega+3] proves \typetrans{\Zr}{\omega+\omega}. This allows Elsa to simulate the set-theoretic hierarchy up to $V_{\omega+\omega}$. Living within $V_{\omega+\omega}$, Elsa can find an uncountable $A$ well-ordered by some relation $<$.\footnote{Note that $A$ is not a von Neumann ordinal, i.e.\ $A$ is not well-ordered by $\in$. Still, the existence of some such $A$ and $<$ follows (without Choice) from Hartog's Lemma; see \textcites[185]{Potter:STP}[92]{Incurvati:CSFM}.} Using this, Elsa can define a theory \PCTT[A], whose type indices are the members of $A$ as ordered by $<$. Since $A$ is uncountable, \PCTT[A] is straightforwardly richer than \PCTT[\omega + \omega + 3]. Moreover, using \PCTT[A], Elsa can simulate a much larger chunk of the set-theoretic hierarchy than $V_{\omega+\omega}$; living within that chunk of the hierarchy, she can find larger well-orders; these can be used to supply the indices for some further development of \PCTT{}\ldots and so on. This seems like a case where ideological-bootstrapping might genuinely achieve something. 

Nevertheless, there are hard limits on what ideological-bootstrapping can achieve. In the simplest case, suppose Elsa insists on starting with \PCTT[n], for some finite $n$, because she is uncertain whether there are infinitely many entities. Since \PCTT[n] only yields (surrogates for) finite well-orders, no amount of ideological-boostrapping from this starting point will allow Elsa to obtain any infinite well-order. So, whenever \citeauthor{LinneboRayo:HOI} try to describe any theory \PCTT[\tau] such that $\tau$ is infinite, Elsa will worry whether the theory \emph{itself} even exists.

The shape of this problem is quite general. Say that $\kappa$ is a \emph{hereditary-point} iff $\kappa$ is an infinite cardinal and everything in $V_\kappa$ is strictly smaller than $\kappa$ (so $\omega$ is the first hereditary-point).\footnote{\label{fn:Vfixedbeth}Formally, $\kappa$ is a hereditary-point iff $\kappa$ is an infinite cardinal and any of these equivalent conditions hold (we leave the reader to prove the equivalences): (1)  $(\forall x \in V_\kappa)|x| < \kappa$; (2) $H_\kappa = V_\kappa$, where $H_\kappa = \Setabs{x}{|\clearme{trcl}(x)| < \kappa}$; (3) either $\kappa = \omega$ or $\kappa$ is a $\beth$-fixed point, i.e.\ $\beth_\kappa = \kappa$; (4)  $|V_\kappa| = \kappa$. Characterisation (1) formalizes the definition in the text; (2) gives the idea its name.} When $\kappa$ is a hereditary-point, it is in principle impossible to ideologically-bootstrap your way from below $\kappa$ to above $\kappa$, since every entity below level $\kappa$ is strictly smaller than $\kappa$ itself. 

This problem is especially pertinent, given two facts about hereditary-points and \ZF. First, \ZF proves that there are proper-class-many hereditary-points; but, since any hereditary-point after $\omega$ would be pretty enormous, it is not unreasonable to wonder whether \emph{any} exist; and ideological-bootstrapping cannot quiet such qualms.\footnote{\textcite[258]{Boolos:MBST} had qualms about the existence of the first $\aleph$-fixed point; calling it $\kappa$, he wrote that $\kappa$ is `so big\ldots\ that it calls into question the truth of any theory, one of whose assertions is the claim that there are at least $\kappa$ objects'. The first hereditary-point after $\omega$ is at least as large as Boolos's $\kappa$; it is a $\beth$-fixed point, as in (3) of footnote \ref{fn:Vfixedbeth}, and hence an $\aleph$-fixed point.} Second, the standard models of \ZF are the $V_\kappa$ such that $\kappa$ is strongly inaccessible; and every strongly inaccessible cardinal is a hereditary-point; so ideological-bootstrapping cannot \emph{possibly} address any ontological worries that an Elsa-like character might have about the existence of any standard model of \ZF. 

The argument of \S\ref{s:LR:ontologyrelocated} therefore stands essentially unchanged. \citeauthor{LinneboRayo:HOI} are mistaken to think that cumulative type theories can help us to overcome ontological worries, since the very existence of the (syntactically individuated) theories themselves requires a rich ontology in the metatheory.

\section{CTT: superfluous type-restrictions}
\label{CTT-unjustified}
In \S\ref{s:GodelSuperfluous}, we discussed Gödel's claim that \STT's type-restrictions were `superfluous'. We should now make explicit something which we there left implicit: these type-restrictions are superfluous \emph{given Gödel's aims}. Specifically, Gödel wanted to establish a foundational, `formal system which avoids the logical paradoxes and retains all [of] mathematics' \parencite*[46]{Godel:PSFM}. Given those aims, \CTT's type-restrictions are just as superfluous as \STT's; it is best to follow Gödel, and work with something like \Zr, with its \emph{untyped} variables.

All of this is compatible with the idea that, given \emph{alternative} aims, \STT's or \CTT's type-restrictions might not be superfluous, but deeply important. As we will show in this section, though, \CTT's type-restrictions are \emph{inevitably} `superfluous restrictions', in the sense that any semantics for \CTT also licenses the use of an \emph{untyped} variable and allows the `types' to be defined internally. (Cf.\ points \diffinternal and \diffuntyped from \S\ref{s:LR:interpret:differences}.) So, in a slogan: \CTT's type-restrictions are superfluous, on any semantics.

We will unpack the details in a moment. First, we should explain the phrase `a semantics for \CTT'. As we are using that phrase, a semantics for \CTT is a general framework within which to provide models of \CTT, rather than a specific model of some \CTT[\tau]. (Compare the idea of `the possible worlds semantics for modal language'.) So, in providing a semantics for \CTT, we fix  the meaning of phrases like `a model of \CTT' and `an entity of type $\alpha$'; the latter will be the sort of entity which, according to the semantics, can be the value of a type $\alpha$ variable. 

\subsection{The abstract argument for introducing untyped variables}
\label{s:types-sense}
Our argument begins with an uncontentious point: the stringently-stated rules for \CTT tell us that $y^\beta(x^\alpha)$ is well-formed iff $\beta > \alpha$; but these rules are needlessly stringent, on any given semantics. 

To see this, fix some semantics for \CTT, and let $\beta \leq \alpha$. The formula $x^\alpha \inCTT y^\beta$ is well-formed according to \CTT. So, for any model $\model{M}$ and any type $\alpha$ entity $a^\alpha$ and type $\beta$ entity $b^\beta$ from $\model{M}$, either $\model{M} \models a^\alpha \inCTT b^\beta$ or $\model{M} \models \lnot a^\alpha \inCTT b^\beta$. (Note: what exactly this comes to will depend on the details of the semantics; but we are proceeding abstractly for now and want to consider \emph{any} semantics for \CTT.) Now, as explained in \S\ref{type theories:CTTlr}, the formula $x^\alpha \inCTT y^\beta$ perfectly simulates the formula $y^\beta(x^\alpha)$; that is, it perfectly simulates the notion of applying a type $\beta$ entity to a type $\alpha$ entity. So we could have allowed $y^\beta(x^\alpha)$ to count as well-formed, even though $\beta \ngtr \alpha$. So, \CTT's stringently-stated type-restrictions are needlessly stringent.

To be clear, this is not an \emph{objection} to \CTT's type-restrictions. We are really just repackaging a point we made in \S\ref{type theories:CTTlr}, and also made by  \textcite[282--3]{LinneboRayo:HOI}, that we can liberalise \CTT's stringently-stated formation rules, and allow that $y^\beta(x^\alpha)$ is well-formed for any type-indices $\alpha$ and $\beta$. From a purely formal point of view, this changes almost nothing. So, in what follows, we will simply \emph{allow} that \CTT counts every formula $y^\beta(x^\alpha)$ as well-formed. 

Significantly, though, \CTT still lacks \emph{untyped} variables. But, for exactly the same reason, this is also needlessly stringent, on any given semantics. 

To see this, fix some semantics for (liberally formulated) \CTT. Now $y^\beta(x^\alpha)$ is well-formed for any $\alpha$ and $\beta$. So, for any model $\model{M}$ and any type $\alpha$ entity $a^\alpha$ and type $\beta$ entity $b^\beta$ from that model, either $\model{M} \models b^\beta(a^\alpha)$ or $\model{M} \models \lnot b^\beta(a^\alpha)$. That is, any model assigns a truth value to the application of any entity to any entity, whatever their types might happen to be. So we could have allowed the \emph{untyped} atomic formula, $y(x)$, to count as well-formed: whatever specific values the variables take, the formula would just amount to applying some entity to some entity, which is exactly what the semantics allows. 

The upshot is that any semantics for \CTT also licenses the use of untyped variables. This time, though, we do have an \emph{objection} to \CTT's type-restrictions. Whereas stringently-formulated \CTT can simulate any \emph{typed}-formula $y^\beta(x^\alpha)$, via $x^\alpha \inCTT y^\beta$, it lacks the technical resources to simulate the \emph{untyped}-formula $y(x)$. Untyped variables have to be added by hand. But, once we have added them, we will have moved from a typed to an untyped theory; if we choose to retain `typed' variables, then they will just behave as restricted untyped variables.

Of course, if there had been no consistent way to introduce untyped variables, then \CTT's type-restrictions would have been far from superfluous. But, in this sort of a context, theories like \Zr provide us with a clear method for consistently introducing untyped variables.\footnote{We do not need all of 
\Zr; we can make do with the subtheory \LT. For details, see \S\ref{s:PCTT:appendix} and \textcite{Button:LT1}.} Moreover, they also provide us with a paradigm for how to define the notion of `type' (i.e.\ \emph{rank}) within the theory. So \CTT's type-restrictions are genuinely superfluous.\footnote{\label{fn:pedantry}\emph{Pedantic Objection.} Perhaps \CTT[\tau]'s externally supplied types are not wholly superfluous, since they allow us to formulate the intrinsically infinitary Limit-rules which gives \CTT[\tau] a kind of strength which a recursive theory like \Zr cannot simulate (see \diffrecursive from \S\ref{s:LR:interpret:differences}). \emph{Pedantic Reply.} Those who \emph{want} to lean on \CTT[\tau]'s infinitary features can incorporate them within a \Zr-like setting. We will illustrate how using \Zr itself. For each index $\alpha < \tau$, introduce a new constant, $c_\alpha$; add to \Zr each sentence `$c_\alpha$ is a von Neumann ordinal'; add the sentence `$c_\alpha \in c_\beta$' iff $\alpha < \beta$; for each limit $\lambda < \tau$, add the infinitary rule: from $\phi(c_\alpha)$ for all $\alpha < \lambda$, infer $(\forall x \in c_\lambda)\phi(x)$.} 

\subsection{Illustration: the class semantics}\label{s:illustrationclass}
The argument of the previous subsection is very abstract. To make it more concrete, in this subsection and the next, we will consider two specific semantics in detail: the class semantics, and the plural semantics. Just as our abstract argument predicts, both semantics clearly license the use of untyped variables.

(To avoid any unfortunate misunderstandings: we offer these semantics merely as illustrations. When we say that no semantics could justify the adoption of \CTT's type-restrictions, we are not making an inductive inference from these two examples; that conclusion was established by the abstract argument of \S\ref{s:types-sense}.)

We start by considering the class semantics. To define a model for \CTT within this semantics, we first specify some suitable set of urelements, $U$. We then stipulate that the type $\alpha$ entities are the members of $U_{\alpha+1}$, where we define:
\begin{align*}
    U_1 & \coloneq U \cup \{\emptyset\} & 
    U_{\alpha+1} & \coloneq \powerset(U_\alpha) \cup U & 
    U_\beta & \coloneq \bigcup_{\alpha<\beta} U_\alpha\text{ for limit }\beta 
\end{align*}
Finally, we offer a general clause governing the semantics of atomic sentences: 
\begin{listclean}
    \item `$b^\beta(a^\alpha)$' is true iff the referent of `$a^\alpha$' is a member of the referent of `$b^\beta$'.
\end{listclean}\noindent
Uncontroversially, \CTT is sound for the class semantics. A stringently-typed formula like `$b^2(a^0)$' will be true (in a model) iff the referent of `$a^0$' is a member of the referent of `$b^2$'. A liberally-typed formula like `$b^0(a^2)$' will also be true (in a model) iff the referent of `$a^2$' is a member of the referent of `$b^0$'; and this will inevitably be false, since the latter is guaranteed to be an urelement, i.e.\ an individual without members.

Our semantic clause for atomic sentences employed type restrictions. However, on the class semantics, the type-restrictions are straightforwardly superfluous. We can easily offer a similar semantic clause for untyped terms:
\begin{listclean}
    \item `$b(a)$' is true iff the referent of `$a$' is a member of the referent of `$b$'.
\end{listclean}
Otherwise put: there is no barrier to introducing untyped variables, whose values can be any individual or class. Of course, given the old paradoxes, we will have to take care in introducing untyped variables. However, as we have already discussed, \Zr-like theories show us how to do this safely. 

\subsection{Illustration: the plural semantics}\label{s:illustrationplural}
The class semantics concerns a class-hierarchy built from a basis of individuals. The plural semantics concerns a plural-hierarchy built from a similar basis.\footnote{\textcite{Rayo:BP} develops this plural interpretation.} In a little more detail, we use the phrase `plural$^*$' as a catch-all for whatever we find at any level in the plural hierarchy, i.e., any object, any objects, any objectses, \ldots, any objects(es)$^\alpha$\ldots.\footnote{\label{fn:pseudosingular}Our word `plural$^*$' is a `pseudo-singular device', in the sense of \textcite[ch.15]{OliverSmiley:PL}; in natural language, it infelicitously behaves like a singular term. \textcite[\S11.8]{FlorioLinnebo:MO} use `higher plurality' here.} We then offer this general clause governing the semantics for atomic sentences:
\begin{listclean}
    \item `$b^\beta(a^\alpha)$' is true iff what `$b^\beta$' refers to includes what `$a^\alpha$' refers to.\footnote{The inclusion is \emph{vertical} in the sense of \textcite[ch.15]{OliverSmiley:PL}. Vertical inclusion only ever holds between plurals$^*$ of different levels, and is analogous to set-membership. Vertical inclusion is to be contrasted with \emph{horizontal} inclusion, which is analogous to subsethood: $b$ horizontally includes $a$ iff $b$ vertically includes everything that $a$ vertically includes.}
\end{listclean}\noindent
So `$b^2(a^0)$' is true iff what `$b^2$' refers to includes what `$a^0$' refers to; and `$b^0(a^2)$' is true iff what `$b^0$' refers to includes what `$a^2$' refers to. But equally, the semantic clause applies perfectly well to untyped terms:
\begin{listclean}
    \item `$b(a)$' is true iff what `$b$' refers to includes what `$a$' refers to.
\end{listclean}
Again: there is no barrier to introducing untyped variables, whose values can be any plural$^*$.

As before, care must be taken to preserve consistency. But we know how to take care: roughly stated, we just need to do for plurals$^*$ what \Zr does for classes/sets. In more detail, instead of setting up a plural$^*$-hierarchy using type-restricted variables with \emph{externally} supplied type-indices, we can reason about plurals$^*$ using an untyped variable, with the plurals$^*$ arranged into a cumulative hierarchy according to their rank (with `rank' defined within the theory, using our untyped variable). And this work has been carried out carefully: \textcites[ch.15]{OliverSmiley:PL}[\S12.6]{FlorioLinnebo:MO} both present consistent plural logics featuring untyped variables. Indeed, Florio and Linnebo develop their untyped plural logic precisely by starting with the \CTT on the plural semantics, and then collapsing the types in the way that we have described.

\section{\STT: type-restrictions justified}
\label{STT-justified}
We have argued that \CTT's type-restrictions are inevitably superfluous. They are unnecessary for the aim of providing a foundational theory for mathematics, and they cannot be justified semantically, since any semantics for \CTT will permit the introduction of an untyped variable.

In this section, we will show that \STT's type-restrictions are not similarly superfluous. We can justify the adoption of \STT's type-restrictions by invoking the \emph{Fregean semantics}. Indeed, on this semantics, a formula is intelligible iff it is well-formed in \STT.

\subsection{Against referentialism}
\label{s:referentialism}

In \S\S\ref{s:illustrationclass}--\ref{s:illustrationplural}, we used the class and plural semantics to illustrate our objection to \CTT. Both of these semantics are \emph{referentialist}. By this we mean that both semantics treat every type of term as a type of \emph{referring} term: every type of term performs the same semantic role---\emph{referring}---and all that changes is what they refer to---individuals, classes/plurals$^*$, or something else.\footnote{We are speaking as if variables \emph{refer}. This is one way to gloss a Tarskian referentialist approach to semantics: the value of a variable (on a Tarskian valuation) can be thought of as the variable's referent (on the valuation). In certain contexts, describing variables as \emph{referring} is misleading (see \cite[ch.1]{ButtonWalsh:PMT}), but we do not think it will do any harm here. If we wanted, we could say that a semantics is \emph{referentialist} iff it treats every type of \emph{constant} as a referring term, and then use a Robinsonian or hybrid approach to handle variables (again, see \cite[ch.1]{ButtonWalsh:PMT}).}

The class and plural semantics render \CTT's type-restrictions superfluous, precisely because they are referentialist. After all, if every type of term performs the same kind of semantic role as every other type of term, then every type of term can be meaningfully substituted for every other type of term. In that case, as we argued in \S\ref{s:types-sense}, the semantics will also allows us to introduce an untyped variable. It follows, immediately, that any semantics which might justify \STT's type-restrictions will have to be non-referentialist; in other words, it will have to assign different kinds of semantic role to different types of term.

Now, at one time, this might have seemed like an impossible demand. According to the old Quinean \parencite*[66--8]{Quine:PL} orthodoxy, we can only quantify into the position of a referring term; so type theory---which allows us to bind variables of every type---must be given a referentialist semantics. Fortunately, times have changed, and philosophers are increasingly willing to accept quantification into other kinds of position.\footnote{See \textcites[ch.\ 3]{Prior:OT}{Boolos:NP}{RayoYablo:ND}[458--60]{Williamson:E}[ch.\ 5]{Williamson:MLM}{Wright:OQPP}{Trueman:PP}.} In what follows, we will simply assume that the old Quinean orthodoxy is mistaken, and will present a particular non-referentialist semantics---the Fregean semantics---which justifies \STT's type-restrictions. 

\subsection{Conceptual but referentialist semantics}
\label{s:conceptualreferentialism}
The Fregean semantics is a variety of \emph{conceptual} semantics. On a conceptual semantics, type theories are theories of predication:\footnote{This point is emphasised throughout \textcite{FlorioJones:UQSTT}.} `$a^0$' is a name which refers to an object; `$b^1$' is a first-level predicate which expresses a property of objects (a \emph{type 1 property});\footnote{We have taken the label `conceptual semantics' from \textcite[272]{LinneboRayo:HOI}, who use `concept' instead of `property'. Of course, \citeauthor{LinneboRayo:HOI} are following Frege here. However, this use of `concept' is potentially misleading; we prefer `property', which avoids any psychological overtones.} `$c^2$' is a second-level predicate which expresses a property of properties of objects (a \emph{type 2 property}); and so on. 

This way of characterising conceptual semantics is schematic, and we get different versions of the semantics when we supply different accounts of what it means for a predicate to \emph{express} a property. On one view of predication, predicates `express' properties in the sense that they \emph{refer} to properties, just as names refer to objects. To illustrate, take the following sentence:
\begin{listn-0}
    \item\label{ex:nc:wise}Socrates pontificates
\end{listn-0}
According to this view of predication, `pontificates' refers to the property \emph{Pontification}.\footnote{This was arguably the standard way of thinking about predication before Frege introduced his alternative (see below), and plenty of philosophers after Frege have advocated versions of it too: see \textcites{Strawson:SPLG}{Strawson:CPPC}[ch.\ 4]{Bealer:QC}{Wiggins:SRP}{Gaskin:BRCUP}{Gaskin:UP}.} 
Clearly, combining this account of predication with the conceptual semantics yields 
another brand of referentialism. Every type of term is still referential; all that changes is whether it refers to an ordinary individual (like Socrates) or to something within a property-hierarchy (like \emph{Pontification}). We then have the following semantic clause for atomic sentences:
\begin{listclean}
    \item `$b^\beta(a^\alpha)$' is true iff the referent of `$a^\alpha$' instantiates the referent of `$b^\beta$'
\end{listclean}\noindent
This allows us to make sense of `$b^\beta(a^\alpha)$', for \emph{any} types $\alpha$ and $\beta$. For example, `$b^0(a^0)$' is true iff the referent of `$a^0$' instantiates the referent of `$b^0$'. Now, admittedly, this formula would correspond to something slightly peculiar in natural language. If `$a^0$' referred to Socrates, and `$b^0$' referred to Plato, then we might try to render `$b^0(a^0)$' as:
\begin{listn}
    \item\label{ex:nc:plato}Socrates Plato
\end{listn}
This is scarcely grammatical English. Still, for referentialists about predication, \eqref{ex:nc:plato} is intelligible: it says that Socrates instantiates Plato. Indeed, precisely this point is made by \textcite{Magidor:LDTC}, who insists that \eqref{ex:nc:plato} is perfectly meaningful and trivially false. We are not agreeing with Magidor here, but we do think that \emph{referentialists} about predication should agree with her. Moreover, and as in \S\ref{s:types-sense}, \emph{referentialists} about predication will ultimately find type-restrictions superfluous; nothing will prevent them from introducing untyped variables and insisting that `$b(a)$' is true iff the referent of `$a$' instantiates the referent of `$b$'.

\subsection{Fregean semantics}
\label{s:Fregean-semantics}
There is, however, a \emph{non-referentialist} version of the conceptual semantics: it is a \emph{Fregean} semantics. 

Unlike referentialists, {Fregeans} do not think that predicates refer to properties (not, at least, in anything like the sense that a name `refers').\footnote{For discussion of the very different sense in which predicates could be said to refer, see \textcite[chs.\ 4--6]{Trueman:PP}.}
Rather, they think that the role of a predicate is to \emph{say something of} an object; for example, `pontificates' says of an object that it pontificates. This is the sense in which Fregeans think that predicates are `incomplete', and they indicate this by writing their predicates with gaps. So rather than writing the predicate in \eqref{ex:nc:wise} as `pontificates', they write it as `\metavariablea{x} pontificates', where `\metavariablea{x}' marks a gap for a name to go. We can then say that sentence \eqref{ex:nc:wise} is true iff `\metavariablea{x} pontificates' says something true of the referent of `Socrates', i.e.\ iff Socrates pontificates.\footnote{This account of predication is what we take to be suggested by Frege's (e.g.\ \cites*{Frege:FC}{Frege:CO}[\S31]{Frege:GA}) discussions of predication; however, we do not want to commit to any exegetical claims here. It is worth noting that the gap between our Fregeans and the referentialists about predication need not be as large as it initially appears. Even if referentialists think of words like `pontificates' as referring terms, on a par with names like `Socrates', concatenation behaves like a Fregean predicate: `$\metavariablea{x}\metavariablea{y}$' says of a pair of objects that the former instantiates the latter. This point is originally due to \textcite[192--3]{Frege:CO}, and is further developed by \textcite[\S\S3.4\,\&\,8.4]{Trueman:PP}.}

From this Fregean perspective, \eqref{ex:nc:plato} is not just ungrammatical, but \emph{unintelligible}. We arrive at it by taking \eqref{ex:nc:wise} and replacing its predicate, `\metavariablea{x} pontificates', with a referring name, `Plato'. Names and predicates are made to work together, but two names cannot work together in the same way. It is not within a name's remit to \emph{say} anything of an object; names just refer to objects. And that is why \eqref{ex:nc:plato} is a meaningless string: neither name says anything of the referent of the other (let alone something true or false).

Now consider the following sentence:
\begin{listn}
    \item\label{ex:nc:somewise} Someone pontificates
\end{listn}
This sentence is not made by combining a predicate with a name. Instead, it is made by combining two predicates, `\metavariablea{x} pontificates' and `Someone \textsf{Y}'. Crucially, though, these are two different types of predicates: `\metavariablea{x} pontificates' is a \emph{first-level} predicate, meaning that `\metavariablea{x}' marks a gap for a name; `Someone \metavariableb{y}' is a second-level predicate, meaning that `\metavariableb{y}' marks a gap for a first-level predicate. Just as first-level predicates play a different kind of semantic role from the names they can take as input, second-level predicates play a different kind of semantic role from the first-level predicates that they can take as input. We might describe this role thus: a second-level predicate says something of things said of objects. This means that \eqref{ex:nc:somewise} is true/false iff `Someone \metavariableb{y}' says something true/false of what `\metavariablea{x} pontificates' says of objects. Specifically: `Someone \metavariableb{y}' says something true of what `\metavariablea{x} pontificates' says of objects iff `\metavariablea{x} pontificates' says something true of someone; and it says something false of what `\metavariablea{x} pontificates' says of objects iff `\metavariablea{x} pontificates' says something false of everyone.

Again, from this Fregean perspective, it is easy to see that we cannot meaningfully substitute a name for the first-level predicate in \eqref{ex:nc:somewise}. Attempting to do this would yield:
\begin{listn}
    \item\label{ex:nc:someplato} Someone Plato
\end{listn}
This string is not just ungrammatical, but meaningless. To be meaningful, the input to `Someone \metavariableb{y}' must be the kind of expression that {says something of} objects. But `Plato' \emph{refers to} an object, rather than \emph{saying anything of} objects (let alone something true of someone or false of everyone). So, if we try to plug `Plato' into the argument-place of `Someone \metavariableb{y}', we end up with garbage.\footnote{\textcite[Introduction, ch.II, \S4]{WhiteheadRussell:PM1} present a similar argument (in their distinctive terminology).}

The crucial point is that, on the Fregean semantics, different types of term play different types of semantic role: `$a^0$' is a name which \emph{refers} to an object; `$b^1$' is a first-level predicate which \emph{says something of} objects; `$c^2$' is a second-level predicate which \emph{says something of things said of objects}; and so on. And rather than having a single semantic clause which applies to all atomic sentences, we have different clauses for different types of predication:
\begin{listclean}
    \item`$b^1(a^0)$' is true iff `$b^1$' says something true of the referent of `$a^0$'
    \item`$c^2(b^1)$' is true iff `$c^2$' says something true of what `$b^1$' says of objects
    \item$\ldots$
\end{listclean}
These semantic clauses allow us to make sense of `$b^n(a^m)$' iff $n=m+1$. This is how the Fregean semantics justifies \STT's type-restrictions: a formula is intelligible on the Fregean semantics iff it is well-formed in \STT. 

For the same reason, the Fregean semantics also prohibits the introduction of untyped variables. Untyped variables would need to be able to take any entity of any type as their values. But, on the Fregean semantics, there is no one sense in which different types of entity could be the `value' of a variable; the sense in which an object is the value of a type 0 variable is incommensurable with the sense in which a type 1 property is the value of a type 1 variable.

To be clear, we are not trying to argue here that anyone should adopt the Fregean semantics.\footnote{For an extended argument to that effect, see \cite{Trueman:PP}.}
Our point here is just that \STT's type-restrictions, unlike \CTT's, are justified by at least one semantics.\footnote{We have considered two conceptual semantics: referentialist and Fregean. \textcites{Wright:GoS}{MacBride:IR}{Liebesman:PA}{Rieppel:BS} offer a third approach, which attempts to provide a \emph{middle-way} between referentialism and Fregeanism. They agree with referentialists that `\metavariablea{x} pontificates' denotes \emph{Pontification}, but they agree with Fregeans that `\metavariablea{x} pontificates' says of objects that they teach. Given the latter point, they agree that first-level predicates play a different kind of semantic role from names; so they agree with Fregeans that `$c^2(a^0)$' is unintelligible. However, unlike Fregeans, they cannot embrace \STT: according to the middle-way, every type $1$ property is also a type $0$ object, but \STT-Comprehension straightforwardly entails that there are strictly more type $1$ properties than objects. Moreover, one of us \parencite[chs.\ 4 \& 8]{Trueman:PP}
has also argued at length that this middle-way is philosophically incoherent.}

\subsection{`Cumulative types' as ambiguous}
\label{conceptual:equivalence}
We have just argued that the Fregean semantics prohibits the introduction of untyped variables. But what it cannot prohibit, of course, is the introduction of \emph{ambiguous} variables, which sometimes behave as one type, and sometimes behave as another. And in fact, this provides the Fregeans with one way of starting to make sense of \CTT. Specifically, they can treat $a^0$ as an ambiguous term: in $b^1(a^0)$, it behave as a \emph{name}, and so refers to an object; in $c^2(a^0)$, it behaves as a \emph{first-level predicate}, and so says something of an object.

If that is how we are to read formulas like $c^2(a^0)$, though, then they no longer represent any departure from \STT. Working in \STT, we can introduce an injective type-raising function, $\uptype$, from objects to type $1$ properties; so $a^0$ is an object, but $\uptype a^0$ is a type $1$ property (We also lay down rules to ensure that $\uptype a^0$ behaves as a suitable surrogate for `the $a^1$ such that $a^1 \eqCTT a^0$'; for details, see \S\ref{equivalence:CTT}.) To avoid ambiguity, we can then rewrite $c^2(a^0)$ as $c^2(\uptype a^0)$, which is now well-formed according to \STT's type-restrictions. 

This idea can be extended across all finite types. The result is \STTu, which augments \STT with a theory of type-raising functions, like $\uptype$, whilst retaining \STT's type-restrictions. We can then prove the following strong result: \CTT[\omega] and \STTu are \emph{definitionally equivalent} (where \CTT[\omega] is the fragment of \CTT which uses all and only finite  type indices; for details, see \S\ref{equivalence:CTT}).

There is, however, an important limitation to this equivalence result. Since entities do not really cumulate in \STTu, \STTu cannot accommodate transfinite types, and so cannot recapture any transfinite uses of \CTT. This is significant, because \citeauthor{LinneboRayo:HOI}'s main argument for \CTT invokes transfinite types (see \S\ref{optimism and union}). For this reason, \citeauthor{LinneboRayo:HOI} must have intended \CTT to be taken at face-value, rather than as a disguised form of \STTu. Unfortunately for them, though, nothing could justify \CTT's type-restrictions, taken at face-value; that was the lesson of \Sref{CTT-unjustified}.

\section{The Semantic Argument}
\label{optimism and union}
We have established an important difference between \CTT and \STT: nothing could justify 
\CTT's type-restrictions, but the Fregean semantics can justify \STT's type-restrictions. In this section, we will respond to \citeauthor{LinneboRayo:HOI}'s Semantic Argument for \CTT. This argument is designed to show that \STT is semantically unstable, and that restoring stability pushes us to \CTT. We will not present any new objections to \CTT in this section; our aim is simply to explain how an advocate of the Fregean semantics should reply to \citeauthor{LinneboRayo:HOI}.

\subsection{Naïve Optimism and Naïve Union}\label{s:naivety}
\citeauthor{LinneboRayo:HOI} introduce us to two notions:
\begin{listbullet}
    \item A \emph{$\beta$-order language} is a language which contains variables of all (and only) the types $\alpha < \beta$.\footnote{\label{fn:constants}It can also contain type $\alpha\mathord{+}1$ constants, for any $\alpha < \beta$.} 
    \item A \emph{generalized semantic theory} for a language is `a theory of all possible interpretations the language might take' \parencite[275]{LinneboRayo:HOI}. In particular, a generalized semantic theory for a $\beta$-order language provides an interpretation which allows any type $\alpha$ entity to be the value of a variable $x^\alpha$, for each $\alpha< \beta$.\footnote{This is very slightly different from what \textcite[300]{LinneboRayo:HOI} actually say: they consider interpretations of constants (see footnote \ref{fn:constants}). The particular requirement on generalized semantic theories is an application of the principle that for each $\alpha$, it is possible to quantify unrestrictedly over all entities of type $\alpha$. (\textcite[274]{LinneboRayo:HOI} only state this principle for type $0$, but their argument requires that the principle apply to all types; \textcite[\S11.5]{FlorioLinnebo:MO} explicitly commit themselves to the fully general principle.) We discuss the broader concept of \emph{absolute generality} in \S\ref{type theories:CTTfj}.}
\end{listbullet}
These notions are connected by two formal results \parencite[Appendix B]{LinneboRayo:HOI}:
\begin{listclean}
    \item[\emph{Blocker Theorem.}] No language can provide a generalized semantic theory for itself.
    \item[\emph{Enabler Theorem.}] For any $\beta$, let $\beta^* = \beta + 2$ if $\beta$ is a limit and $\beta^* = \beta+1$ otherwise; then a $\beta^*$-order language can provide a generalized semantic theory for a $\beta$-order language.
\end{listclean} 
The Blocker Theorem holds by familiar, liar-like reasoning. Moreover, as  \textcite[162--3]{FlorioShapiro:STTTAG} note, it shows that these two principles are jointly inconsistent: 
\begin{listclean}
    \item[\emph{Naïve Optimism.}] Any language can be given a generalized semantic theory.
    \item[\emph{Naïve Union.}] For any languages, there is a union language, which combines all the expressions of those languages.
\end{listclean}
To see the problem: by Naïve Union, there is a language, $\lang{U}$, which is the union of all languages; by Naïve Optimism, $\lang{U}$ can be given a generalized semantic theory in some language $\lang{G}$; by the Blocker Theorem, $\lang{G}$ is not a sub-language of $\lang{U}$; but this contradicts the fact that $\lang{U}$ is the union of all languages, including $\lang{G}$.

\subsection{\citeauthor{LinneboRayo:HOI}'s Semantic Argument}
\citeauthor{LinneboRayo:HOI} avoid contradiction by restricting Naïve Union as follows:
\begin{listclean}
    \item[\emph{Limited Union}.] For any limit $\lambda$, if there is a $\beta$-order language for every  $\beta < \lambda$, then there is also a $\lambda$-order language.\footnote{\textcite[276]{LinneboRayo:HOI}. Note that they also \parencites*[294]{LinneboRayo:HOI}[176]{LinneboRayo:RFS} consider a second, slightly differently restricted principle: \emph{For any `definite totality' of languages, there is a union language}. For our purposes, there is no significant difference between these formulations. \textcite[179--80]{LinneboRayo:RFS} treat `definite totality' as an unanalysed notion. However,  the function of this notion is as follows: given any `definite totality' of languages, we can comprehend a limit-index, $\lambda$, which acts as an upper bound of the orders on the languages among that `definite totality'. (This notion of an `upper bound' makes sense, since every \CTT-like language has well-ordered indices.) So, once we recall that we have only insisted that our type-indices be \emph{well-ordered}, not that they \emph{be ordinals}, the two principles come to the same thing.}
\end{listclean}
Having restricted Naïve Union in this way, \citepossess[275--81]{LinneboRayo:HOI} \emph{Semantic Argument} for \CTT now gets going. Here is a very brief summary. Suppose we start with an ordinary first-order language. By Naïve Optimism, this language has a generalized semantic theory. By the Blocker Theorem, this generalized semantic theory cannot be given in a first-order language; but, by the Enabler Theorem, it can be given in a second-order language. Naïve Optimism now requires that this second-order language has a generalized semantic theory; as before, the Blocker and Enabler theorems will lead us to provide this semantics in a third-order language. This process repeats, running through every finite order. At this point, Limited Union kicks in, giving us an $\omega$-order language which combines all of the finite orders into a single language. To present a generalized semantic theory for this language, Naïve Optimism and the Blocker and Enabler Theorems push us up to an $\omega\mathord{+}2$-order language. And there is now no stopping us: Naïve Optimism, Limited Union and the two theorems keep pushing us to countenance languages of higher and higher orders. Moreover, when we supply the semantics for variables of some limit type $\lambda$, the only plausible option is to allow them to take all entities of all types $<\lambda$ as values. And this requires that at least some of our types be \emph{cumulative}.

\subsection{Rebutting the Semantic Argument}
We agree with the following conditional: \emph{if} we accept both Naïve Optimism and Limited Union, \emph{then} there is good reason to embrace \CTT. Our response is to reject Naïve Optimism (and to insist on Naïve Union). However, we will show that our stance is more principled that \citeauthor{LinneboRayo:HOI}'s.

\textcite[286]{LinneboRayo:HOI} motivate Limited Union as follows: whenever you are `prepared to countenance languages of order $\beta$ for every $\beta < \lambda$', you `should also countenance languages of order $\lambda$', since `they would be made up entirely of vocabulary that had been previously deemed legitimate'. This line of reasoning is compelling. However, it clearly generalizes, to provide a motivation for Naïve Union. After all: whenever you are prepared to countenance some languages, you should also countenance their union, for that union would be made up entirely of vocabulary that had been previously deemed legitimate. In short: the only motivation \citeauthor{LinneboRayo:HOI} offer for Limited Union is \emph{really} a motivation for Naïve Union.

Of course, Naïve Union is inconsistent with Naïve Optimism. So, if there were a stellar argument in favour of Naïve Optimism, we could see the retreat from Naïve Union to Limited Union as a simple instance of the heuristic that, on encountering a contradiction, we should aim to get as close as we can to what we initially wanted, without falling into inconsistency.\footnote{Cf.\ \textcite[485, 492ff.]{Maddy:BAI} on the rules of thumb `\emph{one step back from disaster}' and `\emph{maximize}'; and cf.\ \textcite[274, esp.\ fn.8]{LinneboRayo:HOI} on the rule of thumb: `Because we can.'} Regrettably, though, \citeauthor{LinneboRayo:HOI} provide no argument for Naïve Optimism. So, \emph{prima facie}, an equally good instance of that heuristic would be to accept Naïve Union and instead restrict Naïve Optimism. This threatens to leave us with a deadlock, between those who want to restrict Naïve Union (and so embrace \CTT), and those who want to restrict Naïve Optimism (and so might reject \CTT). 

Fortunately, the argument of \S\ref{STT-justified} provides a principled way to break the deadlock: if we are working with a Fregean semantics for the types, then we should restrict Naïve Optimism. Specifically, we should replace Naïve Optimism with the following:
\begin{listclean}
    \item[\emph{Finite Optimism}.] Any language of any finite order can be given a generalized semantic theory.
\end{listclean}
To be clear: the motivation for this restriction is not simply to avoid contradiction. (As far as restoring formal consistency goes, {Finite Optimism} is serious overkill.) Rather, {Finite Optimism} expresses the exact amount of optimism which is even \emph{coherent} on the Fregean semantics. Since Fregean types cannot cumulate, we cannot make any sense of the idea of an $\omega\mathord{+}2$-order language on the Fregean semantics. {Finite Optimism} and Naïve Union push us to countenance an $\omega$-order language, like \STT itself, but we are pushed no further. Otherwise put: \STT is the \emph{principled} limit on Fregean types.

\section{Partially cumulative types}
\label{type theories:CTTfj}
In this paper, we have critically discussed \CTT, which is the approach to cumulative types favoured by \citeauthor{LinneboRayo:HOI}. In this final section, we will discuss an alternative approach to cumulative types, due to \textcite[\S5]{FlorioJones:UQSTT}. 

\CTT is cumulative in two senses: first, $b^\beta(a^\alpha)$ is well-formed whenever $\beta > \alpha$; second, the values of $x^\beta$ include all of the values of $x^\alpha$, whenever $\beta \geq \alpha$. \citeauthor{FlorioJones:UQSTT}' cumulative type theory---call it \FJT---is cumulative only in the first of these senses. Indeed, for them, no type $\alpha$ entity is a type $\beta$ entity, when $\alpha \neq \beta$. As we will see, this difference between \CTT and \FJT is a double-edged sword: on the one hand, it provides Florio and Jones with the means to defend \FJT from the argument we offered against \CTT in \S\ref{CTT-unjustified}; on the other hand, it leaves so little distance between \FJT and \STT, that \FJT is best seen as a misleadingly formulated version of \STT.

\subsection{\FJT}
\label{s:CTTfj-formal}
Since entities do not cumulate up the types in \FJT, its quantifier rules must be more restrictive than \CTT's (see \S\ref{type theories:CTTlr}). Indeed, \FJT has exactly the same quantifier rules as \STT (see \S\ref{type theories:STT}). Consequently, in \FJT, you cannot generalize about \emph{everything} that has a type $2$ property by writing $\forall x^1(a^2(x^1) \lonlyif \phi(x^1))$.\footnote{Throughout this section, we assume a conceptual semantics, and so speak of type $n > 0$ entities as \emph{properties}. \textcite[45--7]{FlorioJones:UQSTT} offer \FJT as a theory of predication, and we also think that \STT is best understood as a theory of predication.} In \FJT, that formula generalizes over every type 1 property that has $a^2$, but it says nothing about any objects that have it. To cover everything that might have $a^2$, we must conjoin that formula with $\forall x^0(a^2(x^0) \lonlyif \phi(x^0))$. Indeed, to generalize over everything that might have a type $n$ property, we will need $n$ conjuncts. This is illustrated by \citepossessalt[55]{FlorioJones:UQSTT} version of Comprehension:
\begin{listclean}
    \item[\emph{\FJT-Comprehension}.] $\exists z^n\bigland_{i < n}\forall x^i (z^n(x^i) \liff \phi_i(x^i))$, for each $n > 0$, whenever each $\phi_i(x^i)$ is well-formed and does not contain $z^n$.
\end{listclean}
The various $\phi_i$s need have nothing in common, so this is an instance of \FJT-Comprehension:
    $$\exists z^2(\forall x^1(z^2(x^1) \liff x^1=x^1)\land\forall x^0(z^2(x^0) \liff x^0 \neq x^0))$$
As \textcite[61]{FlorioJones:UQSTT} observe, this entails $\forall x^0 \forall y^1\phantom{)} x^0 \neqCTT y^1$, where $\eqCTT$ is defined as before. More generally, in \FJT, if $n \neq m$ then $\forall x^n \forall y^m\phantom{(}x^n\neqCTT y^m$. So \FJT contradicts \CTT's Type-Raising principle (see \S\ref{type theories:CTTlr}).

\subsection{\FJT is finitary}\label{s:CTTfjFinitary}
In formulating \FJT-Comprehension, we have reverted to using natural numbers as type indices, rather than allowing that types might be transfinite (contrast the formulation of \CTT-Comprehension in \S\ref{type theories:CTTlr}). We have done this for a simple reason: formulating \FJT-Comprehension for a transfinite type, $\beta$, would require infinitary conjunction:
    $$\exists z^\beta \bigland_{\alpha < \beta} \forall x^\alpha (z^\beta (x^\alpha) \liff \phi_\alpha(x^\alpha))$$ 
But \FJT does not allow for infinitary conjunction. Consequently, \FJT cannot comprehend any \emph{transfinite} types.\footnote{At least: \citeauthor{FlorioJones:UQSTT} nowhere discuss infinitary conjunction, and only ever use natural numbers as type indices.}

Much of our discussion of \CTT focussed on the Sets-from-Types Theorem (see \S\S\ref{s:LR:interpreting}--\ref{s:DiscussingSetsFromTypesTheorem}). However, due to its finitary nature, \FJT cannot establish any similar result. Indeed, if we add surrogates for purity and extensionality to \FJT, the resulting theory is \emph{decidable}.\footnote{The surrogate for extensionality is the scheme, for all $n > 0$: $\forall x^{n}\forall y^{n}(\bigland_{i < n}\forall z^i(x^{n}(z^i) \liff y^{n}(z^i)) \lonlyif x^{n} = y^{n})$); the surrogate for purity is Type-Purity (see \S\ref{s:formulatingPCTT}). To see that the resulting theory is decidable, note two facts: (i) all its variables are explicitly typed; and (ii) for each $n$, it proves that there are exactly $h(n)$ type $n$ entities, where $h(0) = 1$ and $h(n+1) = 2^{h(0) + \ldots + h(n)}$; it follows that every quantifier provably has a fixed finite range.}

\subsection{Interpreting \FJT's types}
\label{s:CTTfj-ambiguous?}
Having discussed the Sets-from-Types Theorem, we then argued that \CTT's type-restrictions cannot be justified semantically (see \S\ref{CTT-unjustified}). We began with \citepossess[282--3]{LinneboRayo:HOI} observation that, even if we stuck with the stringent formation rules for \CTT, we could always apply $b^\beta$ to $a^\alpha$ in \CTT with the formula $a^\alpha\inCTT b^\beta$, which is defined as follows (where $\gamma = \max(\alpha, \beta)+1$):
\begin{align*}
	a^\alpha \inCTT b^\beta & \mliffdf (\exists x^{\gamma} \eqCTT b^\beta) x^{\gamma}(a^\alpha)
\end{align*}
We then argued that, since every type of entity can be applied to every type of entity in \CTT, there can be no barrier to introducing untyped variables.

This line of argument is not straightforwardly applicable to \FJT. Since entities do not cumulate up the types in \FJT, $b^n$ is not identical to any entity of type $k \neq n$. So, as \textcite[\S7]{FlorioJones:UQSTT} stress, it is doubtful whether $a^m \inCTT b^n$, i.e.\ $(\exists x^{k} \eqCTT b^n) x^{k}(a^m)$ with $k = \max(m,n)+1$, provides us with a way of applying $b^n$ to $a^m$ in \FJT.

Nonetheless, we are still left with the question of how to justify the type-restrictions imposed by \FJT. \textcite[45--7]{FlorioJones:UQSTT} explicitly intend to provide \FJT with some version of the conceptual semantics, but it is unclear which version they could have in mind. The referentialist version that we discussed in \S\ref{s:conceptualreferentialism} licenses the use of an untyped variable; the Fregean version that we discussed in \S\ref{s:Fregean-semantics} justifies \STT's type- restrictions; so it seems that neither of these versions of the conceptual semantics could serve their purpose.

In fact, appearances are somewhat misleading here. It is true that, when \FJT is taken at face value, the Fregean semantics cannot justify its type-restrictions. However, it turns out that the Fregean semantics can make good sense of \FJT, if its terms are interpreted as being systematically ambiguous, in the following way: in `$c^2(b^1)$', `$c^2$' expresses a type $2$ property, but in `$c^2(a^0)$', it expresses a type $1$ property. (Compare the interpretation of \CTT in \STTu of \S\ref{conceptual:equivalence}.)\footnote{Eagle-eyed readers will notice a slight difference between this and \S\ref{conceptual:equivalence}. When dealing with \CTT, we read $c^2(a^0)$ as $c^2(\uptype a^0)$, since \CTT licenses Type-Raising, which projects entities upwards through the levels of the type hierarchy. By contrast, \FJT contradicts Type-Raising; and \FJT-Comprehension effectively projects entities downwards.} 

This ambiguity can easily be handled by augmenting \STT with a theory of type-lowering relations. We start by introducing a type-lowering relation, $\typedownrel$, from type $2$ to type $1$. We then read `$c^2(b^1)$' verbatim, but treat `$c^2(a^0)$' as shorthand for `$\forall x^1 (c^2\typedownrel x^1\rightarrow x^1(a^0))$'. This latter formula is perfectly well-formed according to \STT's type-constraints, and the idea can be extended across all types. The resulting theory is \STTd. We can then prove that \FJT and \STTd are {definitionally equivalent}. (For details and proof, see  \S\ref{equivalence:FJT}.)

We think that \FJT is best understood as a (somewhat misleading) formulation of \STTd. To begin with, there is no obvious reason to resist this interpretation of \FJT. \citeauthor{LinneboRayo:HOI} had a clear technical reason for refusing to interpret \CTT via \STTu: the major selling point of \CTT was meant to be its ability to accommodate transfinite types (see \S\ref{optimism and union}). But, as we saw in \S\ref{s:CTTfjFinitary}, \FJT is just as limited to finite types as \STT. So \FJT, like \STT, cannot go beyond \emph{Finite Optimism}.  

Not only is there no reason for \citeauthor{FlorioJones:UQSTT} to resist the interpretation of \FJT as \STTd, there is good reason for them adopt it. Their \parencite*{FlorioJones:UQSTT} main aim is to argue that cumulative type theories can accommodate \emph{absolute generality}. However, as we will now show, \FJT can accommodate absolute generality iff it is taken as a mere notational variant of \STTd.

\subsection{\STT accommodates absolute generality}
\label{absolute generality:STT}

We start by explaining how \STT accommodates absolute generality.

In traditional set-theoretic semantics, domains are taken to be sets. In \STT, we can think of them as properties. For example, we can think of a domain of objects as a type 1 property, $d^1$, and we can say that $x^0$ is in that domain iff $d^1(x^0)$. As \textcite{Williamson:E} clearly explains, there is a real advantage to thinking of domains in this type-theoretic way. There is no set of all objects, and so if we think of domains as sets, unrestricted quantification over all objects is impossible. But \STT straightforwardly supplies a type $1$ property, $\bigdomain^1$, held by all objects, i.e.:\footnote{Via $\exists z^1 \forall x^0(z^1(x^0)\liff x^0=x^0)$, which is an instance of \STT-Comprehension and \CTT- and \FJT-Comprehension.}
    $$\forall x^0 \bigdomain^1(x^0)$$
(Nothing special is signified by our use of a capitalized `$\bigdomain$' here; it simply aids readability.)

Whilst $\bigdomain^1$ includes all the objects, one might worry that it is still \emph{restricted}, since it includes no type $1$ properties. But, in the context of \STT, this worry is toothless; no sense can be made of this idea in \STT. To regard $\bigdomain^1$ as \emph{restricted}, we would have to be able to make sense of the idea of a more inclusive domain, which contains both objects and properties.\footnote{We are not saying that there would have to \emph{be} a more inclusive domain, only that it would have to \emph{make sense} to say that there is. If $\phi$ make sense, then so must $\lnot \phi$.} But that is incoherent in \STT. To say `$d$ contains both objects and properties' is to say $\exists x^0\exists y^1(d(x^0) \land d(y^1))$, which is just ungrammatical in \STT. For $d(x^0)$ to be grammatical, $d$ must be type $1$; for $d(y^1)$ to be grammatical, $d$ must be type $2$; but every term has a unique type.

Suppose, then, we introduce suitably typed domains, $d^1$ and $d^2$. In \STT, these domains are \emph{incommensurable}, to use \citepossess[458]{Williamson:E} phrase. This does not mean that $d^1$ and $d^2$ have different members; it means that we cannot even \emph{express} the idea that they have the same (or different) members. We might put this by saying that, in \STT, we cannot articulate a univocal notion of \emph{Thing} or \emph{Entity} which applies to both objects and properties. (We can still talk about `type $0$ entities', `type $1$ entities', etc., but we cannot think of `entity' as a recurring categorematic component in these constructions.) So, if a first-order quantifier quantifies over all objects, then it quantifies over absolutely every \emph{thing} it makes sense to imagine that it might quantify over.

We can put the same point slightly differently by drawing on \citepossessalt[49]{FlorioJones:UQSTT} explication of \emph{unrestrictedness}: `an unrestricted domain is a domain such that true universal quantification over it precludes there from being absolutely any counterexamples whatsoever.'\footnote{This explication has an obvious shortcoming: it employs unrestricted quantification itself, in talking about `absolutely any counterexamples'. However, this shortcoming is shared by {every} account of unrestricted quantification. Moreover, anyone who already understands unrestricted quantification should agree with \citeauthor{FlorioJones:UQSTT}' explication.} This informal explication can be converted into a formal definition in \STT. To say that `everything melted' is {true over} the domain \emph{things in the freezer} is just to say that $\forall x(x\text{ is in the freezer}\lonlyif x\text{ melted})$. More generally, to say that $\forall x^{n-1} y^{n}(x^{n-1})$ is \emph{true over} domain $d^{n}$ is just to say $\forall x^{n-1}(d^{n}(x^{n-1}) \lonlyif y^{n}(x^{n-1}))$. To say that there are absolutely no counterexamples to this restricted generalization is to say that the generalization still holds good even when we lift the restriction, and return to $\forall x^{n-1} y^{n}(x^{n-1})$. And finally, to say that there are absolutely no counterexamples to \emph{any} true quantification over $d^{n}$ is just to generalize over all $y^{n}$. Assembling this, we obtain, for all $n > 0$: 
   \begin{listn-0}
       \item\label{def:STT:unrestricted} $d^{n}$ is unrestricted $\mliffdf\\\phantom{indent}\forall y^{n}(\forall x^{n-1}(d^{n}(x^{n-1})\lonlyif y^{n}(x^{n-1}))\lonlyif \forall x^{n-1}y^{n}(x^{n-1}))$
    \end{listn-0}
This definition is adequate because, in \STT, only generalizations of the form $\forall x^{n-1} y^{n}(x^{n-1})$ can be true over $d^{n}$. And $\bigdomain^1$, as introduced at the start of this subsection, is unrestricted according to \eqref{def:STT:unrestricted}: since $\forall x^0\bigdomain^1(x^0)$, if $\forall x^0(\bigdomain^1(x^0)\lonlyif y^1(x^0))$, it immediately follows that $\forall x^0 y^1(x^0)$. More generally, within \STT, it is obvious that $d^{n}$ is unrestricted iff $\forall x^{n-1} d^{n}(x^{n-1})$.

\subsection{Absolute generality in \FJT}
\label{absolute generality:CTTfj}
We have seen that \STT can accommodate absolute generality. So, if we read \FJT as a (misleadingly formulated) notational variant of \STTd, then \FJT can equally accommodate absolute generality. But, as we will now show, \FJT cannot accommodate absolute generality if it is taken at face-value.

To establish this, we will assume in what follows that \FJT is to be taken at face-value, so that $d^2(y^1)$ and $d^2(x^0)$ apply the very same type $2$ property to $y^1$ and $x^0$. (That assumption will remain in force until we explicitly lift it in \S\ref{s:CTTfj-summary}.) So understood, \FJT allows type $2$ properties to serve as domains containing both objects and type 1 properties. In fact, \FJT delivers a domain, $\bigdomain^2$, which contains all type $1$ properties and all objects, i.e.\ such that:\footnote{Via the \FJT-Comprehension instance: $\exists z^2(\forall x^1(z^2(x^1) \liff x^1=x^1)\land\forall x^0(z^2(x^0) \liff x^0 = x^0))$.}
$$\forall x^1\bigdomain^2(x^1) \land \forall x^0\bigdomain^2(x^0)$$
But now first-order quantification becomes a form of \emph{restricted} quantification: in a clear sense, $\bigdomain^1$ is a restriction of $\bigdomain^2$, since $\bigdomain^2$ contains everything in $\bigdomain^1$, and more besides.\footnote{Formally: $\forall x^0(\bigdomain^1(x^0)\lonlyif \bigdomain^2(x^0))$, but $\exists y^1(\bigdomain^2(y^1)\land \forall x^0(\bigdomain^1(x^0) \lonlyif x^0 \neqCTT y^1))$.}

The point here is that \FJT \emph{does} treat objects and type $1$ properties as a species of a single genus. Indeed, for each $n >0$, we can think of  \emph{Thing}$^{n}$ as the property $\bigdomain^{n}$ such that $\bigwedge_{m < n} \forall x^m\bigdomain^{n}(x^m)$. So in \FJT, it makes sense, \emph{and is true}, to say that first-order quantifiers quantify over some things but not others.\footnote{\textcite{Kramer:ETS} presents a very similar argument, but directed against \CTT rather than \FJT.}

Again, we can make the same point in terms of \citeauthor{FlorioJones:UQSTT}' idea that $d^n$ is unrestricted iff there are absolutely no counterexamples to any universal generalization which is true over $d^n$. Applied to \FJT, this does not quite yield a simple \emph{definition} of unrestrictedness,\footnote{This is because we can ask whether $d^n$ is $m$-unrestricted for any $m > 0$; see \eqref{def:CTTfj:m-unrestricted}, below.} but it does yield a schematic necessary condition for unrestrictedness: if $d^{n}$ is unrestricted, and $y^{m}$ is true of everything in $d^{n}$ that it can be meaningfully applied to, then $y^m$ is true of absolutely everything it can be meaningfully applied to. Formalizing this intuitive idea, we obtain the following, for all $m,n > 0$:
    \begin{listn}
        \item \label{def:CTTfj:unrestricted} $d^{n}$ is unrestricted $\lonlyif\\\phantom{indent}\forall y^{m}(\bigland_{i < \min(m, n)}\forall x^{i}(d^{n}(x^{i})\lonlyif y^{m}(x^{i}))\lonlyif \bigland_{i < m}\forall x^{i} y^{m}(x^{i}))$
    \end{listn}
This makes $\bigdomain^1$ {restricted}, since \FJT yields an $H^2$ which applies to every object but to no type $1$ property, i.e.\ such that:\footnote{Via the \FJT-Comprehension Instance: $\exists z^2(\forall x^0(z^2(x^0)\liff x^0 = x^0) \land \forall x^1(z^2(x^1)\liff x^1\neq x^1))$.}
    $$\forall x^1\lnot H^2(x^1) \land \forall x^0H^2(x^0)$$
Clearly $\forall x^0(\bigdomain^1(x^{0}) \lonlyif H^2(x^0))$, but $\lnot \forall x^1 H^2(x^1)$; so $\bigdomain^1$ is restricted by \eqref{def:CTTfj:unrestricted}.\footnote{This informal argument crucially assumes that \FJT is taken at face value. Take the idea that `$y^2$ is true of everything in $d^1$ that it can meaningfully be applied to' can be glossed in \FJT as $\forall x^0(d^1(x^0) \lonlyif y^2(x^0))$. Under interpretation into \STTd, this formula becomes $\forall x^0(d^1(x^0) \lonlyif \forall z^1(y^2 \typedownrel z^1 \lonlyif z^1(x^0)))$. This no longer says anything about whether $y^2$ itself is true of everything in $d^1$.} A similar argument shows that every domain of every type is restricted in \FJT.\footnote{Assuming that the type hierarchy does not have a terminal level.} (And the same style of argument shows that no domain is unrestricted in \CTT.)\footnote{When \CTT is taken at face value, and again assuming that the type hierarchy does not have a terminal level. In detail: first, we observe that if $d^{\beta+1}$ is unrestricted, then $\forall y^{\gamma+1}(\forall x^\alpha(d^{\beta+1}(x^{\alpha})\lonlyif y^{\gamma+1}(x^{\alpha}))\lonlyif \forall x^{\gamma}y^{\gamma+1}(x^{\gamma}))$, for all $\alpha \leq \min(\beta, \gamma)$. Via \CTT-Comprehension, we obtain an $H^{\beta+2}$ such that $\forall x^{\beta+1}(H^{\beta+2}(x^{\beta+1})\liff \exists x^\beta\phantom{(}x^{\beta+1} \eqCTT x^{\beta})$. Then any $d^{\beta+1}$ is restricted, since $\forall x^\beta(d^{\beta+1}(x^\beta) \lonlyif H^{\beta+2}(x^\beta))$ but $\lnot \forall x^{\beta+1} H^{\beta+2}(x^{\beta+1})$.}

\subsection{\citeauthor{FlorioJones:UQSTT} on (R=U)}
\label{absolute generality:relatively unrestricted}
Our argument that every domain is restricted in \FJT was based on \citeauthor{FlorioJones:UQSTT}' own explication of \emph{unrestrictedness}. But they thought that \FJT could accommodate absolute generality. In this subsection, we will lay out their reasoning, and explain why it was mistaken.

Alongside their explication of \emph{unrestrictedness}, \textcite[51]{FlorioJones:UQSTT} introduce a further notion: a domain is \emph{Russellian} for a generalization $\forall v Fv$ iff it coincides with the \emph{range of significance} of the predicate $F$, i.e.\ the range of things that $F$ can be meaningfully applied to. They then propose  \parencite*[51--3]{FlorioJones:UQSTT}:
\begin{listbullet}
    \item[(R=U)]A domain is Russellian iff it is unrestricted.
\end{listbullet}
Here is the idea behind (R=U): a counterexample to $\forall v Fv$ would be something of which $F$ is false; but $F$ does not say \emph{anything} (whether true or false) of the things which fall outside of its range of significance; so if $\forall v Fv$ is true over $d$, and $d$ is Russellian for $\forall v Fv$, then there cannot be any counterexamples to $\forall v Fv$; so $d$ is unrestricted for $\forall v Fv$.

\textcite[57--8]{FlorioJones:UQSTT} attempt to use (R=U) as follows. The domain $\bigdomain^1$ is Russellian for the generalization $\forall x^0 a^1(x^0)$: after all, type 1 terms express type 1 properties, and type 1 properties apply meaningfully only to objects.\footnote{This explains why \citeauthor{FlorioJones:UQSTT} abandoned \citeauthor{LinneboRayo:HOI}'s \CTT, in favour of a  theory which invalidates Type-Raising: in \CTT, every type of entity can be applied to every type of entity (using $\inCTT$ if necessary), and so the range of significance of $a^1$ includes all entities of all types.} So if we read $\forall x^0 a^1(x^0)$ as a quantification over $\bigdomain^1$, then by (R=U) it is \emph{unrestricted}. Whilst $\bigdomain^1$ is a strict sub-domain of $\bigdomain^2$, none of the extra entities in $\bigdomain^2$ fall within $a^1$'s range of significance.

Our basic problem with (R=U) is quite simple: there is a fundamental mismatch between the R and the U. \emph{Unrestrictedness} is normally understood in absolute terms: either a domain is absolutely unrestricted, or it is not. By contrast, \citeauthor{FlorioJones:UQSTT}'s notion of \emph{Russellianness} is a relative matter: a domain is not just Russellian \emph{full stop}; it is only ever Russellian \emph{for} a generalization $\forall v Fv$. This basic problem can be overcome in \STT, but not in \FJT.

In \STT, a property $d^{n}$ can (meaningfully) be a domain for, and only for, generalizations of the form $\forall x^{n-1} y^{n}(x^{n-1})$. After all, if we attempt to relativize the generalization $\forall y^{m}(x^i)$ to  $d^{n}$, obtaining $\forall x^i(d^{n}(x^i) \lonlyif y^{m}(x^i))$, then the result is grammatical in \STT iff $n = m = i+1$. 
Consequently, the relativity involved in \emph{Russellianness} can be safely ignored: it would not even make sense to say that $d^{n}$ is Russellian for $\forall x^{m-1} y^{m}(x^{m-1})$ when $n\neq m$. Indeed, since the range of significance of any type $n$ property in \STT is always exactly the type $n\mathord{-}1$ entities, we can say that $d^{n}$ is Russellian $\mliffdf$ $\forall x^{n-1} d^{n}(x^{n-1})$. Using \eqref{def:STT:unrestricted} from \S\ref{absolute generality:STT}, we can then prove (R=U) for \STT.

In \FJT, by contrast, a property $d^n$ can (meaningfully) be a domain for generalizations $\forall x^i y^m(x^i)$ with $n \neq m$, so we cannot simply ignore the relativity in \emph{Russellianess}. Let us, then, try to accommodate it. Officially, a domain is supposed to be Russellian for a \emph{generalization}. However, since the range of significance of any type $m$ property in \FJT is always exactly the type $k < m$ entities, all that really matters is the type of the predicate used in the generalization. This leads to an explicitly relativized notion of \emph{Russellianness} as follows:
\begin{listn}
    \item \label{def:CTTfj:m-russellian} $d^{n}$ is ${m}$-Russellian $\mliffdf$ 
    all and only the type $k < m$ entities have $d^{n}$
    \item[i.e.{}] 
    $\bigl(\bigland_{k < m}\forall y^k \biglor_{i < n} (\exists x^i \eqCTT y^k)d^{n}(x^i)\bigr) \land \bigl(\bigland_{i < n}\forall x^{i}(d^{n}(x^{i}) \lonlyif \biglor_{k < m}\exists y^{k}\ x^{i} \eqCTT y^{k})\bigr)$
\end{listn}
The first conjunct captures the idea that every type $k < m$ entity has $d^{n}$; it says that every type $k < m$ entity is an entity in $d^{n}$. The second conjunct captures the idea that only the type $k < m$ entities have $d^{n}$; it says that every (type $i < n$) entity in $d^{n}$ is a type $k < m$ entity.\footnote{\textcite[56fn.15]{FlorioJones:UQSTT} have some doubts about whether $\eqCTT$ expresses cross-type identity in \FJT. If these doubts are justified, then our formal definition of \emph{$m$-Russellian} will have to be revised as follows:
\begin{listn}
    \item[(\ref{def:CTTfj:m-russellian}$_*$)] $d^{n}$ is ${m}$-Russellian$_*$ $\mliffdf$ $\bigl(\bigland_{k<m}\forall y^kd^n(y^k)\bigr) \land \bigl(\bigland_{m\leq k < n}\forall y^k\lnot d^n(y^k)\bigr)$
\end{listn}
If $n < m$, then `$d^n$ is $m$-Russellian$_*$' is ill-formed rather than false. Nevertheless, our key points about \emph{Russellianness} still go through. First, Russellianness$_*$ is significantly relativized, since $H^2$ is $1$-Russellian$_*$ but not $2$-Russellian$_*$, with $H^2$ as given at the end of \S\ref{absolute generality:CTTfj}. Second, (R=U) is false, since $\bigdomain^2$ is $1$-unrestricted but not $1$-Russellian$_*$, with $\bigdomain^2$ as given in \S\ref{absolute generality:CTTfj}.} This definition allows us (meaningfully) to ask whether $d^{n}$ is $m$-Russellian, for any $n$ and $m$. Furthermore, if $n < m$, then $d^{n}$ is {not} $m$-Russellian. In particular, $U^1$ is not $2$-Russellian. However, $U^1$ {is} $1$-Russellian. So, in \FJT, \emph{Russellianness} is significantly relativized.

To make sense of (R=U) in \FJT, then, \citeauthor{FlorioJones:UQSTT} must relativize the notion of \emph{unrestrictedness}, so that it matches the relativity in \emph{Russellianness}. Tacitly, they do exactly this, describing domains as unrestricted \emph{for} certain generalizations \parencite*[e.g.\ 52--3]{FlorioJones:UQSTT}. \citeauthor{FlorioJones:UQSTT} do not define this relative sense of `unrestricted', but we can easily provide a definition on their behalf. To say that $d^{n}$ is unrestricted with regard to type $m$ is, presumably, to say this: if $y^{m}$ is true of everything in $d^{n}$ that it can be meaningfully applied to, then it is true of absolutely everything it can be meaningfully applied to. Formalizing this, we obtain the following, for all $m,n >0$:
    \begin{listn}
        \item \label{def:CTTfj:m-unrestricted} $d^{n}$ is $m\text{-unrestricted} \mliffdf \\\phantom{indent}\forall y^{m}(\bigland_{i < \min(m, n)}\forall x^{i}(d^{n}(x^{i})\lonlyif y^{m}(x^{i}))\lonlyif \bigland_{i < m}\forall x^{i} y^{m}(x^{i}))$
    \end{listn}
Indeed, this just turns \eqref{def:CTTfj:unrestricted}, which is a schematic necessary condition on unrelativized unrestrictedness, into a definition of relativized $m$-unrestrictedness.

We can now understand (R=U) thus: a domain is $m$-Russellian iff it is $m$-unrestricted. But so understood, (R=U) is false: $\bigdomain^2$ is $1$-unrestricted but not $1$-Russellian, with $\bigdomain^2$ as given in \S\ref{absolute generality:CTTfj}. Moreover, we do not need  any principle like (R=U) to determine whether a given domain is $m$-unrestricted; we can just use definition \eqref{def:CTTfj:m-unrestricted}. For example, it is clear from \eqref{def:CTTfj:m-unrestricted} that $\bigdomain^1$ is $1$-unrestricted but $2$-restricted. More generally, $d^{n}$ is $m$-unrestricted iff both $n \geq m$ and $\bigland_{i < m}\forall x^i d^{n}(x^i)$.

The only remaining question is whether the salient notion of \emph{unrestrictedness} in \FJT is the absolute notion governed by \eqref{def:CTTfj:unrestricted}, or the relative notion defined by \eqref{def:CTTfj:m-unrestricted}. We think it is completely clear that the relevant notion is the absolute one. After all, the debate here is about \emph{absolute generality}. It would be false advertising to enter that debate, promising to vindicate unrestricted quantification, and then only deliver \emph{relatively} unrestricted quantification. To emphasise this point, return to the example of $\bigdomain^1$: evidently, $\bigdomain^1$ is $1$-unrestricted but $2$-restricted, as defined by \eqref{def:CTTfj:m-unrestricted}. Precisely because $\bigdomain^1$ is $2$-restricted, though, there is a clear sense in which $\bigdomain^1$ is restricted \emph{simpliciter}. In particular, with $H^2$ as given at the end of \S\ref{absolute generality:CTTfj}, everything which is $\bigdomain^1$ is $H^2$, i.e.\ $\forall x^0(\bigdomain^1(x^0) \lonlyif H^2(x^0))$, but some entities are not $H^2$, in that $\exists x^1 \lnot H^2(x^1)$.

Indeed, this is exactly where \textcite[57]{FlorioJones:UQSTT} go wrong. They recognise that you can find a type $1$ entity not in $\bigdomain^1$, but say: `it does not entail that $F$ is meaningfully predicable of that entity', where $\forall x^0 F (x^0)$ is the generalization under consideration. However, \FJT has \emph{precisely} that entailment when $F$'s type is $> 1$, as in the case of $F = H^2$.

\subsection{\FJT: the case for ambiguity}
\label{s:CTTfj-summary}

It might be helpful to end our discussion of \FJT by summarizing our case for reading it as a mere notational variant of \STTd.

\emph{First.} We see no reason not to read \FJT in this way. \citeauthor{LinneboRayo:HOI} could not read \CTT[\omega] as a notational variant of \STTu, because they wanted to extend \CTT[\omega] into the transfinite. But \FJT is as finitary as \STTu.

\emph{Second.} If we take \FJT at face-value, then it is unclear how we should interpret it. \citeauthor{FlorioJones:UQSTT} explicitly intended to give \FJT a conceptual semantics, but we know of no version of that semantics which could justify \FJT's type-restrictions, taken at face-value.

\emph{Third.} If we take \FJT at face-value, then it cannot accommodate absolute generality. However, if we read \FJT as a notational variant of \STTd, then it can supply absolutely unrestricted domains. 

\section{Conclusion}
\label{conclusion}
In this paper, we have argued for four main claims:
\begin{listl-0}
	\item\label{conc-1} \CTT cannot be used to close the gap between an ideological hierarchy of types and an ontological hierarchy of sets (\S\S\ref{s:LR:interpreting}--\ref{s:DiscussingSetsFromTypesTheorem}).
	\item\label{conc-2} \CTT's type-restrictions are superfluous, on any semantics (\S\ref{CTT-unjustified}).
	\item\label{conc-3} \STT's type-restrictions can be justified by a Fregean semantics, which also provides us with a way to resist \citeauthor{LinneboRayo:HOI}'s {Semantic Argument} in favour of \CTT (\S\S\ref{STT-justified}--\ref{optimism and union}).
	\item\label{conc-4} \FJT is best understood as a misleading formulation of \STTd (\S\ref{type theories:CTTfj}).
\end{listl-0}
We start with \eqref{conc-1}. The Sets-from-Types Theorem allows us to simulate \Zr within \CTT. But deep mathematical differences remain between \Zr and and \Zrkappa, rendering \Zrkappa unsuitable as a framework for mathematical foundations. Furthermore, the Sets-from-Types Theorem cannot allay any ontological worries we might have about set theory: \CTT's type-indices are supplied externally, and so the Sets-from-Types Theorem merely shunts our ontological worries into the metalanguage.

Next is \eqref{conc-2}. \CTT is a remarkably relaxed type theory: it allows us to apply every type of entity to every type of entity. But it still retains the constraint that all of its variables are typed and, in \CTT, that type-restriction is superfluous. Once every type of entity can be applied to every type of entity, there can be no barrier to introducing untyped variables.

We come now to \eqref{conc-3}. The strict type-restrictions imposed by \STT can be justified by the Fregean semantics. On this semantics, different types of term play fundamentally different types of semantic role, so that they cannot be meaningfully intersubstituted. Moreover, this semantics yields a principled reason to reject \emph{Naïve Optimism}, a crucial premise in \citeauthor{LinneboRayo:HOI}'s Semantic Argument.

We end with \eqref{conc-4}. \citeauthor{FlorioJones:UQSTT}' \FJT was meant to be a \emph{partially} cumulative type theory, but we argue it is best understood as a notational variant of \STTd: taking \FJT at face-value leaves it unable to accommodate absolute generality; whereas \STTd---which is definitionally equivalent to \FJT---provides absolutely unrestricted domains of quantification.

\startappendix

\section{Elementary facts about CTT}\label{s:app:elementary}
The remainder of this paper comprises technical appendices, covering the formal results mentioned in the main text. We will start with some elementary observations about \CTT. As mentioned in \S\ref{type theories:CTTlr}, for each ordinal $\tau$, we have a theory \CTT[\tau].\footnote{Throughout the appendices, we will assume that all type-indices are ordinals; nothing turns on this, but it makes the technicalities more familiar.} Recall that we have explicitly defined $\eqCTT$ and $\inCTT$, 
for any types $\alpha$ and $\beta$ and where $\gamma = \max(\alpha, \beta)+1$:
\begin{align*}
	a^\alpha \eqCTT b^\beta & \mliffdf \forall x^{\gamma}(x^{\gamma}(a^\alpha) \liff x^\gamma(b^\beta))\\
    a^\alpha \inCTT b^\beta & \mliffdf (\exists x^{\gamma} \eqCTT b^\beta)x^{\gamma}(a^\alpha)
\end{align*}
In what follows, we will frequently invoke the following simple facts about $\eqCTT$ and $\inCTT$; 
crudely, they  allow us to move seamlessly between different type-levels:
\begin{lem}\label{fact:TR} If $\alpha \leq \beta$ and $\beta + 1 < \tau$, then \CTT[\tau] proves: $\forall a^\alpha \exists b^\beta\phantom{)}a^\alpha \eqCTT b^\beta$
\end{lem}
\begin{proof}
    By \AllI{\beta+1}{\beta+1}, we have $\forall x^{\beta+1}(x^{\beta+1}(a^\alpha) \liff x^{\beta+1}(a^\alpha))$, i.e.\ $a^\alpha \eqCTT a^\alpha$; so $\forall a^\alpha \exists b^\beta\  a^\alpha \eqCTT b^\beta$ by \ExistsI{\beta}{\alpha} followed by \AllI{\alpha}{\alpha}
\end{proof}
\begin{lem}\label{fact:CTT} For any $\phi$, and any $\alpha, \beta, \gamma$ with $\max(\alpha, \beta, \gamma) + 2 < \tau$, \CTT[\tau] proves: 
	\begin{listn-0}
		\item\label{eqCTT:E} if $a^\alpha \eqCTT b^\beta$ and $\phi(a^\alpha)$, then $\phi(b^\beta)$, when this is well-formed
		\item\label{eqCTT:equiv}
		if $a^\alpha \eqCTT b^\beta \eqCTT c^\gamma$, then $a^\alpha \eqCTT c^\gamma$
		\item\label{inCTT:in} if $a^\alpha \eqCTT b^\beta $ and $a^\alpha \inCTT c^\gamma$, then $b^\beta \inCTT c^\gamma$
		\item\label{inCTT:ni} if $a^\alpha \eqCTT b^\beta$ and $c^\gamma \inCTT a^\alpha$, then $c^\gamma \inCTT b^\beta$
	\end{listn-0}
\end{lem}
\begin{proof}
    \emphref{eqCTT:E} Suppose $a^\alpha \eqCTT b^\beta$ and $\phi(a^\alpha)$. Let $\delta = \max(\alpha, \beta)$; by \CTT-Comprehension there is some $c^{\delta+1}$ such that $\forall x^{\delta}(c^{\delta+1}(x^{\delta}) \liff \phi(x^{\delta}))$. Since $\phi(a^\alpha)$, by \AllE{\delta}{\alpha} we have that $c^{\delta+1}(a^\alpha)$. Since $a^\alpha \eqCTT b^\beta$, i.e.\ $\forall z^{\delta+1}(z^{\delta+1}(a^\alpha) \liff z^{\delta+1}(b^\beta))$, by \AllE{\delta+1}{\delta+1} we have that $c^{\delta+1}(b^\beta)$. Now $\phi(b^\beta)$ by \AllE{\delta}{\beta}. 
    
    \emph{(\ref{eqCTT:equiv})--}\emphref{inCTT:ni} We leave these to the reader. They are not completely immediate consequences of \eqref{eqCTT:E}, since the definitions of $\eqCTT$ and $\inCTT$ are typically ambiguous.
\end{proof}
\begin{lem}\label{lem:Foundation}
    If $\max(\alpha, \beta) + 2 < \tau$, then \CTT[\tau] proves: 
    $a^\alpha \inCTT b^{\beta+1} \liff (\exists x^\beta \eqCTT a^\alpha) b^{\beta+1}(x^\beta)$   
\end{lem}
\begin{proof}
    By Type-Founded and Lemmas \ref{fact:TR}--\ref{fact:CTT}.
\end{proof}\noindent
It is worth emphasising that Type-Founded and Type-Basis are independent from \CTT's other axioms. To show this, we begin by building an ill-founded set-theoretic structure, $\model{A}$. Let $\textbf{a} = \{\emptyset, \textbf{a}\}$; now define:
    \begin{align*}
        A_1 &\coloneq \textbf{a} & 
        A_{n+1} &\coloneq \powerset(A_n) & 
        A & \coloneq \bigcup_{n < \omega} A_n
    \end{align*}
So $A_2 = \{\emptyset, \{\emptyset\}, \{\textbf{a}\}, \textbf{a}\}$. Let $\model{A}$ be the structure whose domain is $A$ and which interprets $\in$ verbatim; evidently, $\model{A}$ is ill-founded. Using a slight tweak of the class semantics of \S\ref{s:illustrationclass}, we now create a model, $\model{M}$, of \CTT[\omega] without Type-Founded. We start by defining a ranking function $\rho : A \functionto \mathbb{N}$ on $\model{A}$ as follows: 
    \begin{align*}
        \rho(\emptyset) &= 0 & 
        \rho(\textbf{a}) &= 1 & 
        \rho(c) = n\text{ iff }c \in A_{n} \setminus A_{n-1}
    \end{align*}
So $\rho(\{\emptyset\}) = 2$. Now we stipulate that $\model{M}$'s type $n$ entities are all those $c \in A$ such that $\rho(c) \leq n$, and applications are stipulated to hold as follows, for all $m < n$ and all $b, c \in A$:
    \begin{align*}
        \model{M} \models c^n(b^m)\text{ iff }b \in c
    \end{align*}
It is easy to confirm that $\model{M}$ models \CTT[\omega] without  Type-Founded. But, by construction, $c^m \eqCTT c^n$ whenever $\min(m, n) \geq \rho(c)$. So $\textbf{a}^2(\textbf{a}^1)$ with $\textbf{a}^2 \eqCTT \textbf{a}^1$, and hence $\textbf{a}^1 \inCTT \textbf{a}^1$. So $\model{M}$ violates Type-Founded. Admittedly, Type-Basis holds in $\model{M}$, but we can violate it with a similar construction: start with a Quine atom $\textbf{b} = \{\textbf{b}\}$; let $B_0 = \textbf{b}$ and $B_{n+1} = \powerset(B_n)$; define $\rho(c) = n$ iff $c \in B_n \setminus B_{n-1}$; and note that $\textbf{b}^0 \inCTT \textbf{b}^0$.

\section{Obtaining \Zr in \PCTT}\label{s:PCTT:appendix}
In \S\ref{s:LR:interpreting}, we stated the Sets-from-Types Theorem. In this appendix, we prove that result. We also introduce the interpreting theory, \PCTT, and the interpreted theory, \Zr, and discuss how \PCTT deals with Replacement.

\subsection{The theory \Zr}\label{s:formulatingZr}
The set theory which we simulate is \Zr. We can think of \Zr as arising by adding to \Zermelo the principle that the sets are arranged in well-ordered levels; \Zr is therefore strictly stronger than \Zermelo\ and strictly weaker than \ZF.\footnote{\Zr is equivalent to \citepossess{Potter:STP} theory \textbf{Z}; this is strictly stronger than Zermelo's \Zermelo.} We follow \citepossess{Button:LT1}'s formulation of \Zr, starting with a core of definitions:
\begin{define}\label{def:widen}\label{def:history}\label{def:level}
	Say that $h$ is a \emph{history}, written $\histpred(h)$, iff $(\forall a \in h)\forall x(x \in a \liff (\exists c \in h)x \subseteq c \in a)$. Say that $s$ is a \emph{level}, written $\levpred(s)$, iff $\exists h(\histpred(h) \land \forall x(x \in s \liff \exists c(x \subseteq c \in h)))$.\footnote{\emph{Notation}: we let `$x \subseteq c \in h$' abbreviate `$(x \subseteq c \land c \in h)$'; similarly for other infix predicates.}
\end{define}\noindent
Using these definitions, we can consider some axioms:
\begin{listaxiom}
	\labitem{Extensionality}{ext} $\forall a \forall b (\forall x(x \in a \liff x \in b) \lonlyif a =b)$
	\labitem{Separation}{sep} $\forall a \exists b \forall x(x \in b \liff (\phi(x) \land x \in a))$, for every $\phi$ not containing $b$
	\labitem{Stratification}{lt:strat} $\forall a (\exists s \supseteq a)\levpred(s)$
	\labitem{Endless}{lt:cre} $\forall a \exists b\ a \in b$
	\labitem{Infinity}{lt:inf} $\exists a(\exists x\ x \in a \land (\forall x \in a)\exists y(x \in y \in a))$
\end{listaxiom}
The theory \LT has, as axioms, \ref{ext}, all instances of \ref{sep}, and \ref{lt:strat}, which serves as a principle of foundation. The theory \Zr adds \ref{lt:cre} and \ref{lt:inf} to \LT. In what follows, these next two results will be extremely useful:\footnote{See \textcite[\S3]{Button:LT1} for proofs.}
\begin{lem}\label{lem:es:acc} \ref{ext} + \ref{sep} proves: 
    if $\levpred(s)$, then $s = \Setabs{x}{\exists r(\levpred(r) \land x \subseteq r \in s)}$.
\end{lem}
\begin{thm}\label{thm:es:wellorder} \ref{ext} + \ref{sep} proves: the levels are well-ordered by $\in$, i.e.:
	\begin{listn-0}
		\item $\exists s(\levpred(s) \land \phi(s)) \lonlyif \exists s(\levpred(s) \land \phi(s) \land \lnot (\exists r \in s)(\levpred(r) \land \phi(r)))$
		\item $\forall s\forall t((\levpred(s) \land \levpred(t)) \lonlyif (s \in t \lor s = t \lor t \in s))$
	\end{listn-0}
\end{thm}\noindent
This last result allows us to define the rank of $a$, written $\text{rank}(a)$, within \LT, in terms of the $\in$-least level with $a$ as a subset.

\subsection{The theory \PCTT}\label{s:formulatingPCTT}
The theory \PCTT extends \CTT with two new principles.\footnote{Compare these with \textcite[149 \emph{Ext}, 153 \emph{Nullity}]{DegenJohannsen:CHOL}.} First, we add a version of `extensionality', for all $\alpha \leq \beta$:
\begin{listclean}
    \item[\emph{Type-Ext}.]
    $\forall a^{\alpha+1} \forall b^{\beta+1}([\forall x^\alpha(a^{\alpha+1}(x^\alpha) \lonlyif b^{\beta+1}(x^\alpha)) \land {}$\\
    \phantom{.}\hspace{7em}$\forall x^\beta(b^{\beta+1}(x^\beta) \lonlyif (\exists y^\alpha \eqCTT x^\beta)a^{\alpha+1}(y^\alpha))] \lonlyif a^{\alpha+1} \eqCTT b^{\beta+1})$
\end{listclean}
Second, to achieve `purity', we add an axiom stating that there is exactly one object:
\begin{listclean}
    \item[\emph{Type-Purity}.] $\forall x^0\forall y^0\ x^0 = y^0$
\end{listclean}\noindent
Note that, modulo \CTT's other axioms, Type-Founded follows from {Type-Ext} and {Type-Purity}.

To begin our simulation of \Zr within \PCTT, we will show that \PCTT[\tau] proves \typetrans{\ref{ext}}{\kappa} and \typetrans{\ref{sep}}{\kappa}.
\begin{lem}\label{lem:CTT:ext}$\PCTT[\tau] \proves \typetrans{\ref{ext}}{\kappa}$, whenever $\kappa + 2 < \tau$.
\end{lem}
\begin{proof}
    Suppose $\kappa$ is a limit (the proof is easier when $\kappa$ is a successor).   Without loss of generality, fix $\alpha \leq \beta < \kappa$ and suppose $\forall x^\kappa(x^\kappa \inCTT a^\alpha \liff x^\kappa \inCTT b^\beta)$. 
    
    Using Lemma \ref{fact:TR}, find $a^{\alpha+1} \eqCTT a^\alpha$ and $b^{\beta+1} \eqCTT b^\beta$. Suppose $a^{\alpha+1}(x^\alpha)$. So $x^\alpha \inCTT a^{\alpha+1}$ by Lemma \ref{lem:Foundation}; by Lemma \ref{fact:TR} there is $x^\kappa \eqCTT x^\alpha$, and $x^\kappa \inCTT a^{\alpha}$ by Lemma \ref{fact:CTT}; so $x^\kappa \inCTT b^\beta$, and now $b^{\beta+1}(x^\alpha)$ by Lemmas \ref{fact:CTT} and Lemma \ref{lem:Foundation}. Similar reasoning shows: if $b^{\beta+1}(x^\beta)$ then $(\exists y^\alpha \eqCTT x^\beta)a^{\alpha+1}(y^\alpha)$. By Type-Ext, $a^{\alpha+1} \eqCTT b^{\beta+1}$; hence $a^\alpha \eqCTT b^\beta$ by Lemma \ref{fact:CTT}. Generalizing, for any $\alpha, \beta < \kappa$:
    	$$\forall x^\kappa(x^\kappa \inCTT a^\alpha \liff x^\kappa \inCTT b^\beta) \lonlyif a^\alpha \eqCTT b^\beta$$
	Now \typetrans{\ref{ext}}{\kappa} holds, using Limit$^\kappa$ twice.
\end{proof}
\begin{lem}$\PCTT[\tau] \proves \typetrans{\ref{sep}}{\kappa}$, whenever $\kappa + 2 < \tau$.
\end{lem}
\begin{proof}
    Suppose $\kappa$ is a limit (the proof is easier when $\kappa$ is a successor).    Fix $\alpha < \kappa$ and $\phi$ such that $\phi(x^\kappa)$ is well-formed. Fix $a^\alpha$ and find $a^{\alpha+1} \eqCTT a^\alpha$ by Lemma \ref{fact:TR}. Using \CTT-Comprehension, fix $b^{\alpha+1}$ such that: 
		$$\forall x^\alpha(b^{\alpha+1}(x^\alpha) \liff (\forall x^\kappa \eqCTT x^\alpha)(\phi(x^\kappa) \land x^\kappa \inCTT a^{\alpha+1}))$$
	Suppose $z^\kappa \inCTT b^{\alpha+1}$; by Lemma \ref{lem:Foundation} there is $z^\alpha \eqCTT z^\kappa$ such that $b^{\alpha+1}(z^\alpha)$; so using the biconditional, $\phi(z^\kappa) \land z^\kappa \inCTT a^{\alpha+1}$. 
	Conversely, suppose $\phi(z^\kappa) \land z^\kappa \inCTT a^{\alpha+1}$; by Type-Founded there is $z^\alpha \eqCTT z^\kappa$, and $(\forall x^\kappa \eqCTT z^\alpha)(\phi(x^\kappa) \land x^{\kappa} \inCTT a^{\alpha+1})$ by Lemma \ref{fact:CTT}; so $b^{\alpha+1}(z^\alpha)$, and $z^\kappa \inCTT b^{\alpha+1}$ by Lemma \ref{lem:Foundation}. Summarizing: $z^\kappa \inCTT b^{\alpha+1} \liff (\phi(z^\kappa) \land z^\kappa \inCTT a^{\alpha+1})$. By Lemma \ref{fact:TR} there is $b^\kappa \eqCTT b^{\alpha+1}$. Generalizing and using Lemma \ref{fact:CTT}, for any $\alpha < \kappa$: 
		$$\forall a^\alpha \exists b^\kappa \forall z^\kappa(z^\kappa \inCTT b^\kappa \liff (\phi(z^\kappa) \land z^\kappa \inCTT a^\alpha))$$
	Now \typetrans{\ref{sep}}{\kappa} follows by the  Limit$^\kappa$-rule.
\end{proof}\noindent
Consequently, \PCTT[\tau] proves Lemma \typetrans{\ref{lem:es:acc}}{\kappa} and Theorem \typetrans{\ref{thm:es:wellorder}}{\kappa}. The latter result states that the $\levpred^{(\kappa)}$s are well-ordered by $\inCTT$. Here, `$\levpred^{(\kappa)}$' is the obvious translation of the definition of `$\levpred$'; we also call these levels$^{(\kappa)}$. In what follows, we also write things like $x^\kappa \subseteq^{(\kappa)} y^\kappa$ for $(\forall v^\kappa\inCTT x^\kappa)v^\kappa \inCTT y^\kappa$. 

Our next goal is to show that \PCTT[\tau] simulates our set-theoretic principle of foundation, i.e.\ \ref{lt:strat}. 
We first need a small subsidiary lemma, which says (roughly) that any subset of a low-typed entities is itself low-typed:
\begin{lem}\label{lem:CTT:subset}
    $\PCTT[\tau] \proves \forall a^\alpha (\forall b^\kappa \subseteq^{(\kappa)} a^\alpha)\exists x^\alpha\ b^\kappa \eqCTT x^\alpha$, whenever $\kappa + 2 < \tau$.
\end{lem}
\begin{proof}
    Suppose $\alpha$ and $\kappa$ are limits (the proof is easier otherwise). Let $\beta < \kappa$, and fix $b^\beta \subseteq^{(\kappa)} a^\alpha$; it suffices to show that $\exists x^\alpha\ b^\beta \eqCTT x^\alpha$, since we can then use Limit$^\kappa$ to establish the result. 
    
    If $\beta \leq \alpha$, Lemma \ref{fact:TR} immediately tells us that $\exists x^\alpha\ b^\beta \eqCTT x^\alpha$. Suppose instead that $\beta > \alpha$. Fix $\gamma < \alpha$, and suppose there is some $a^{\gamma+1} \eqCTT a^\alpha$. Using Lemma \ref{fact:TR}, let $b^{\beta+1} \eqCTT b^\beta$. By \CTT-Comprehension, there is $c^{\gamma+1}$ such that: 
        $$\forall v^\gamma(c^{\gamma+1}(v^\gamma) \liff b^{\beta+1}(v^\gamma))$$
    Using Lemmas \ref{fact:TR}--\ref{lem:Foundation}: if $b^{\beta+1}(v^\beta)$, then $v^\beta \inCTT a^{\gamma+1}$ since $b^\beta \subseteq^{(\kappa)} a^{\gamma+1}$, so that there is $y^\gamma \eqCTT v^\beta$; now $b^{\beta+1}(y^\gamma)$, so that $c^{\gamma+1}(y^\gamma)$. Generalizing, $\forall v^\beta(b^{\beta+1}(v^\beta) \lonlyif (\exists y^\gamma \eqCTT v^\beta)c^{\gamma+1}(y^\gamma))$. By Type-Ext, $c^{\gamma+1} \eqCTT b^{\beta+1} \eqCTT b^\beta$. Summarizing all this, we have established the following conditional, for each $\gamma < \alpha$:
        $$\exists x^{\gamma+1}\ a^\alpha \eqCTT x^{\gamma+1} \lonlyif \exists x^{\gamma+1}\ b^\beta \eqCTT x^{\gamma+1}$$
    Now, for reductio, suppose that $\forall x^\alpha\ b^\beta \neqCTT x^\alpha$. Then $\forall x^{\gamma+1}\ b^\beta \neqCTT x^{\gamma+1}$ for all $\gamma < \alpha$. So, by the relevant conditional, $\forall x^{\gamma+1}\ a^\alpha \neqCTT x^{\gamma+1}$. By the Limit$^\alpha$-rule, $\forall x^\alpha\ a^\alpha \neqCTT x^\alpha$, a contradiction. Discharging the reductio, $\exists x^\alpha\ b^\beta \eqCTT x^\alpha$, as required.
\end{proof}
\begin{lem}\label{lem:CTT:strat}
	$\PCTT[\tau] \proves \typetrans{\ref{lt:strat}}{\kappa}$, whenever $\kappa + 2 < \tau$.
\end{lem}
\begin{proof}
    We will show that, for each $\beta \leq \kappa$, \PCTT[\tau] proves $\forall a^\beta(\exists s^\beta \supseteq^{(\kappa)} a^\beta)\levpred^{(\kappa)}(s^\beta)$. This is an induction on $\beta$ in the metatheory, where our induction hypothesis is that for each $\alpha < \beta$ we have established $\forall a^\alpha(\exists s^\alpha \supseteq^{(\kappa)} a^\alpha)\levpred^{(\kappa)}(s^\alpha)$
    
	\emph{Induction case when $\beta = 0$.} By Type-Founded, $\forall x^\kappa\ x^\kappa \notinCTT a^0$. So $\histpred^\kappa(a^0)$ and $\levpred^{(\kappa)}(a^0)$, vacuously. So $\forall a^0(\exists s^0 \supseteq^{(\kappa)} a^0)\levpred^{(\kappa)}(s^0)$. 
	
	\emph{Induction case when $\beta$ is a limit.} Applying Lemmas \ref{fact:TR}--\ref{fact:CTT} to our induction hypothesis, we have $\forall a^\alpha(\exists s^\beta \supseteq^{(\kappa)} a^\alpha)\levpred^{(\kappa)}(s^\beta)$. Now $\forall a^\beta(\exists s^\beta \supseteq^{(\kappa)} a^\beta)\levpred^{(\kappa)}(s^\beta)$, by the Limit$^\beta$-rule.
	
	\emph{Induction case when $\beta = \alpha + 1$.} Using \CTT-Comprehension twice, find $h^{\beta}$ and $s^{\beta}$ such that
	\begin{align}
	    \forall x^\alpha&(h^{\beta}(x^\alpha) \liff \levpred^{\kappa}(x^\alpha)) \label{eq:h:hist}\\
	    \forall x^\alpha&\phantom{(}s^{\beta}(x^\alpha) \label{eq:s:lev}
	\end{align}
    Combining these with the induction hypothesis, we obtain:
	\begin{align*}
	    \forall x^\alpha (s^{\beta}(x^\alpha) & \liff (\exists c^\alpha \supseteq^{(\kappa)} x^\alpha)h^{\beta}(c^\alpha)) 
	\end{align*}
    Hence, by Lemmas \ref{fact:TR}--\ref{lem:Foundation} and \ref{lem:CTT:subset}:
    \begin{align}
        \forall x^\kappa &(x^\kappa \inCTT s^{\beta} \liff \exists c^\kappa(x^\kappa \subseteq^{(\kappa)} c^\kappa \inCTT h^{\beta}))         \label{strat:slev} 
	\end{align}
	Next, applying \AllE{\kappa}{\alpha} and \AllI{\alpha}{\alpha} to Lemma \typetrans{\ref{lem:es:acc}}{\kappa} gives:
	\begin{align*}
	    \forall a^\alpha&(\levpred^{(\kappa)}(a^\alpha)  \lonlyif \forall x^\kappa(x^\kappa \inCTT a^\alpha \liff \exists c^\kappa(\levpred^{(\kappa)}(c^\kappa) \land x^\kappa \subseteq^{(\kappa)} c^\kappa \inCTT a^\alpha)))
	\end{align*}
	So, by \eqref{eq:h:hist} and Lemmas \ref{fact:TR}--\ref{lem:Foundation}:
	\begin{align*}
		(\forall a^\kappa \inCTT h^{\beta})\forall x^\kappa(x^\kappa \inCTT a^\kappa \liff (\exists c^\kappa \inCTT h^{\beta})x^\kappa \subseteq^{(\kappa)} c^\kappa \inCTT a^\kappa) \label{strat:shist} 
	\end{align*}
	i.e.\ $h^\beta$ is a history$^{(\kappa)}$. So $s^\beta$ is a level$^{(\kappa)}$, by \eqref{strat:slev}. Moreover, for any $a^{\beta}$, we have $a^\beta \subseteq^{(\kappa)} s^\beta$ by \eqref{eq:s:lev} and Lemmas \ref{fact:TR}--\ref{lem:Foundation}.  So $\forall a^{\beta}(\exists s^{\beta}\supseteq^{(\kappa)} a^\alpha)\levpred^{(\kappa)}(s^{\beta})$. 
\end{proof}\noindent
We have now established \typetrans{LT}{\kappa}. To obtain \typetrans{Zr}{\kappa}, we need just two straightforward results, which we leave to the reader (they hold using \CTT-Comprehension, Lemmas \ref{fact:TR}--\ref{lem:Foundation}, and the Limit-rule).
\begin{lem}
	$\PCTT[\tau] \proves \typetrans{\ref{lt:cre}}{\kappa}$, whenever $\kappa + 2 < \tau$ and $\kappa$ is a limit
\end{lem}
\begin{lem}\label{lem:CTT:inf}
 	$\PCTT[\tau] \proves \typetrans{\ref{lt:inf}}{\kappa}$, whenever $\kappa > \omega$ and $\kappa + 2 < \tau$
\end{lem}\noindent
Assembling Lemmas \ref{lem:CTT:ext}--\ref{lem:CTT:inf}, we have the Sets-from-Types Theorem:
\begin{thm}\label{thm:PCTT:Zr}
    $\PCTT[\tau] \proves \typetrans{Zr}{\kappa}$ for any limit $\kappa > \omega$ with $\kappa + 2 < \tau$.
\end{thm}

\subsection{Replacement, and semantic considerations}\label{s:LTandPCTTmodels}
We mentioned that \Zr sits strictly between \Zermelo and \ZF. Specifically, \Zr does not include Replacement. To settle the status of Replacement with regard to \PCTT,\footnote{This addresses \textcite[289 n.28]{LinneboRayo:HOI}.} we will move from proof theory to semantics, linking models of \PCTT with models of \LT. (Recall that \LT is the subtheory of \Zr whose axioms are \ref{ext}, \ref{sep} and \ref{lt:strat}.)

In considering models of \LT, we restrict our attention to transitive models.\footnote{All the set-theoretic facts needed in this ensuing discussion of transitive models can be found in \textcites[ch.8]{ButtonWalsh:PMT}. \emph{Notation}: We use calligraphic fonts for structures, and italics for their underlying domains; so $A$ is the domain of $\model{A}$. The definition of a transitive model is given in the model theory; so we use `$\in$', here, in the model theory, and use `$\in^\model{A}$' for $\model{A}$'s interpretation of $\in$'} Recall that a structure $\model{A}$ in the signature of set theory is \emph{transitive} iff both $(\forall x \in A)x\subseteq A$, and $(\forall a \in A)(\forall b \in A)(a \in b \liff a \in^\model{A} b)$. So membership and subsethood are absolute for transitive models. Also recall that being a (von Neumann) ordinal is absolute for transitive models,\footnote{Whenever we speak of ordinals in this subsection, we mean von Neumann ordinals.} and so is the notion of a set's (ordinal) rank. (Recall from \S\ref{s:formulatingZr} that we can define a set's rank within \LT, and hence within \Zr.) Where $\model{A}$ is a transitive model of $\LT$, let $\ordmodel{\model{A}}$ be the least ordinal not in $A$ itself. 

Whilst we consider only transitive models of \LT, we will entertain non-standard models. A transitive model $\model{A} \models \LT$ is \emph{standard} iff for any $\alpha < \ordmodel{\model{A}}$, every subset of $\Setabs{x \in A}{\rankof(x) \leq \alpha}$ is itself in $\model{A}$.

Given any model $\model{M} \models \PCTT[\tau]$ with $\kappa + 2 < \tau$, we can easily turn it into a transitive set-theoretic model, $\setmodel{\model{M}}{\kappa}$, as follows: let $\setmodel{\model{M}}{\kappa}$'s domain comprise all the type $\kappa$ entities from $\model{M}$; and let $\setmodel{\model{M}}{\kappa} \models a \in b$ iff $\model{M} \models a \inCTT b$. 
\begin{lem}\label{lem:CTT:mostowski}
    When $\kappa + 2 < \tau$: if $\model{M} \models \PCTT[\tau]$, then $\setmodel{\model{M}}{\kappa}$ is isomorphic to a unique transitive model of \LT. 
\end{lem}
\begin{proof}
    By Lemmas \ref{lem:CTT:ext}--\ref{lem:CTT:strat}, $\setmodel{\model{M}}{\kappa} \models \LT$. The type indices are well-ordered. By the Limit-rule, Type-Founded, Type-Purity and Lemma \ref{lem:CTT:ext}, $\setmodel{\model{M}}{\kappa}$'s membership relation is extensional and well-founded. Now use Mostowski's Collapsing Lemma.
\end{proof}\noindent
We can also move in the opposite direction, from transitive models of \LT to models of \PCTT. In effect, we follow the class-semantics of \S\ref{s:illustrationclass}, but tweaked to ban urelements and to allow for non-standard models of \PCTT, where a model of \PCTT is \emph{standard} iff for any entities of any type $\alpha$ (other than the greatest) in the model, some type $\alpha\mathord{+}1$ property in the model applies exactly to those entities.\footnote{See \textcite[279n.13]{LinneboRayo:HOI}.} Still, the basic plan is simple: start with a transitive model of \LT; treat entities of different rank as being of different types; and read membership as `application'. 

Unfortunately, there is a small wrinkle in implementing this plan, thanks to an irritating mismatch between the types of \PCTT and a set's rank. To illustrate: \PCTT's Limit-rule means that every type $\omega$ entity is of some finite type, but the ordinal $\omega$ has rank $\omega$. To deal with this wrinkle, we define a function which (in effect) tells us how to map from ranks to types:
\begin{align*}
    \ordtypeshunt{\alpha} &=
    \begin{cases}
        \alpha&\text{if }\alpha < \omega\\
        \alpha+1&\text{if }\alpha \geq \omega
    \end{cases} 
\end{align*}
We can now implement our plan. Where $\model{A}$ is a transitive model of \LT, define $\typemodel{\model{A}}$ as follows. Its denizens are just the members of $A$, and if $\rankof(x) = \alpha$ then $x$ is treated as a type $\beta$ entity for all $\ordtypeshunt{\alpha} \leq \beta < \ordtypeshunt{\ordmodel{\model{A}}}$. Then, we stipulate that $\typemodel{\model{A}} \models y^{\alpha}(x^{\alpha})$ iff $\model{A} \models x \in y$. 
\begin{lem}\label{lem:Zr:transitiveiso}
    Let $\model{A} \models \LT$ be transitive. Then $\typemodel{\model{A}} \models \PCTT[\ordtypeshunt{\ordmodel{\model{A}}}]$. Moreover, $\model{A}$ is standard iff $\typemodel{\model{A}}$ is standard. 
\end{lem}
\begin{proof}[Proof sketch]
    The quantifier-rules and Limit-rules are obviously sound. When $\eqCTT$ and $\inCTT$ are well-defined \CTT[\ordtypeshunt{\ordmodel{\model{A}}}]-expressions,\footnote{Recall from footnote \ref{fn:kappa+2explanation} that $x^\alpha \inCTT y^\beta$ is a \CTT[\tau]-formula iff $\max(\alpha, \beta) + 2 < \tau$.} distinct entities $a^{\alpha}$ and $b^\beta$ are distinguished by $\{a^\alpha\}^{\alpha+1}$, so $\eqCTT^{\typemodel{\model{A}}}$ is identity and $\inCTT^{\typemodel{\model{A}}}$ is membership. Type-Base and Type-Purity now hold, as $\model{A}$ has exactly one rank-$0$ object, and it is empty. Type-Ext follows from \ref{ext} and simple reasoning about ranks. For \CTT-Comprehension, fix $\phi$ and $\alpha$ with $\alpha+1 < \ordtypeshunt{\ordmodel{\model{A}}}$; we will  show that: 
        $$\typemodel{\model{A}} \models \exists z^{\alpha+1} \forall x^\alpha(z^{\alpha+1}(x^\alpha) \liff \phi(x^\alpha))$$
    Let $\ordtypeshunt{\beta} = {\alpha+1}$; note that $\beta \in A$, as $\model{A}$ is transitive. Fix $s \in A$ such that $\model{A}$ thinks that $s$ is the $\in$-least level with $\beta$ as a subset. By \ref{sep} on $s$ in $\model{A}$, there is some $c \in A$ of rank $\leq \beta$ which serves as a witnessing value for $z^{\alpha+1}$ when regarded as an entity of type $\alpha+1$. Finally, the remark about standardness is immediate from the construction.
\end{proof}\noindent
We now have the means to move between transitive models of \LT and models of \PCTT. Recalling that \LT is strictly weaker than \ZF, we can now settle the status of Replacement, in \PCTT, by using some well-known facts concerning models of \ZF: 
\begin{thm}\label{thm:ReplacementStatus}
    Fix $\kappa > \omega$ such that $\kappa + 2 < \tau$:
    \begin{listn-0}
      \item\label{k:strongin:rep} If $\kappa$ is strongly inaccessible, every model of \PCTT[\tau] satisfies \typetrans{\ZF}{\kappa}.
      \item\label{k:notstrongin:fails} If $\kappa$ is not strongly inaccessible, there are models of \PCTT[\tau] which violate \typetrans{\ZF}{\kappa}, and any model of \PCTT[\tau] which satisfies \typetrans{\ZF}{\kappa} is non-standard.
     \end{listn-0}
\end{thm}
\begin{proof}
    \emphref{k:strongin:rep}     
    Let $\kappa$ be strongly inaccessible with $\model{M} \models \PCTT[\tau]$. Using Theorem \ref{thm:PCTT:Zr} and Lemma \ref{lem:CTT:mostowski}, obtain a transitive model $\model{A} \isomorphic \setmodel{\model{M}}{\kappa} \models \Zr$. To show that $\model{M} \models \typetrans{\ZF}{\kappa}$, it suffices to show that $\model{A} \models \text{Replacement}$. So: fix $a \in A$ and suppose $\model{A} \models (\forall x \in a)\exists!y \phi(x,y)$. Working outside $\model{A}$, let 
        $$\beta = \sup\Setabs{\rankof(c)}{\model{A} \models (\exists x \in a)\phi(x,c)}$$ 
    Since $\model{A}$ is transitive and $\kappa$ is strongly inaccessible, $\beta \in A$. So also $\Setabs{c}{\model{A} \models (\exists x \in a)\phi(x,c)} \in A$, by Separation in $\model{A}$ on what $\model{A}$ thinks is a level with $\beta$ as a subset. 
    
    \emphref{k:notstrongin:fails} 
    Here are two general facts about the $V_\alpha$ hierarchy:
    \begin{align*}
        V_\kappa \models \ZF&\text{ iff }\kappa\text{ is strongly inaccessible}\\
        V_{\kappa}\models \LT&\text{ iff }\kappa > 0
    \end{align*}
    So suppose $\kappa > \omega $ is not strongly inaccessible. Fix $\sigma$ such that $\ordtypeshunt{\sigma} \geq \tau$. Now, $V_\sigma$ is a transitive model of \LT and $\ordmodel{V_\sigma} = \sigma$. Using Lemma \ref{lem:Zr:transitiveiso}, obtain $\typemodel{V_{\sigma}} \models \PCTT[\tau]$. By construction, $\setmodel{\typemodel{V_{\sigma}}}{\kappa} = V_{\kappa}$. Since $V_{\kappa} \nmodels \text{Replacement}$, also $\typemodel{V_{\sigma}} \nmodels \typetrans{Replacement}{\kappa}$. 
    
    For the second clause: suppose $\model{M} \models \PCTT[\tau]$ and $\model{M} \models \typetrans{\ZF}{\kappa}$, with $\kappa$ not strongly inaccessible. Use Lemma \ref{lem:CTT:mostowski} to obtain a unique transitive model $\model{A} \isomorphic \setmodel{\model{M}}{\kappa}$. Since $\model{A} \models \text{\ZF}$ but $\ordmodel{\model{A}}$ is not strongly inaccessible, $\model{A}$ is non-standard. 
\end{proof}

\section{Definitional equivalence for \ensuremath{\clearme{CTT}^{\omega}}}
\label{equivalence:CTT}
In \S\ref{conceptual:equivalence}, we stated that \CTT[\omega] is definitionally equivalent to \STTu. In this appendix, we define \STTu, and prove the equivalence.

\STTu augments \STT with a new function symbol, $\uptype$, for each $n$, which takes a type $n$ entity as input and outputs a type $n\mathord{+}1$ entity.\footnote{As with the signs $=$, $\eqCTT$ or $\inCTT$, we are using the same symbol (in a typically ambiguous way) for each type level.} So, for example, $b^2(\uptype a^0)$ and $c^5(\uptype\uptype b^2)$ are well-formed. \STTu retains \STT-Comprehension; this holds for formulas containing $\uptype$. \STTu then has axioms ensuring that $\uptype$ is injective, preserves property-possession, and delivers well-foundedness:
\begin{listbullet}
    \item[\emph{Up-Inject.}] $\forall x^n \forall y^n(\uptype x^n = \uptype y^n \lonlyif x^n = y^n)$
    \item[\emph{Up-Possess.}] $\forall x^n \forall y^{n+1}(\uptype y^{n+1}(\uptype x^n) \liff y^{n+1}(x^n))$
    \item[\emph{Up-Founded.}] 
    $\forall x^{n+1}\forall y^{n+1}(\uptype y^{n+1}(x^{n+1}) \lonlyif \exists z^n\ x^{n+1} = \uptype z^n)$   
    \item[\emph{Up-Base.}] $\forall x^0\forall y^0 \lnot \uptype y^0(x^0)$
\end{listbullet}
For readability, where $m > n$, we write $\uptypeto[n]{m}a^n$ for the result of applying $m - n$ instances of $\uptype$ to $a^n$, yielding a type $m$ entity; so $\uptypeto[0]{4} a^0$ abbreviates $\uptype\uptype\uptype\uptype a^0$, and  $c^5(\uptypeto[2]{4}b^2)$ abbreviates $c^5(\uptype\uptype b^2)$. Simple induction, which we leave to the reader, shows that \STTu proves generalizations of our new axioms; specifically, for each $m > n$: 
\begin{listbullet}
    \item $\forall x^n \forall y^n(\uptypeto[n]{m} x^n = \uptypeto[n]{m} y^n \lonlyif x^n = y^n)$
    \item $\forall x^n \forall y^{n+1}(\uptypeto[n+1]{m+1} y^{n+1}(\uptypeto[n]{m} x^n) \liff y^{n+1}(x^n))$
    \item $\forall x^{m}\forall y^{n+1}(\uptypeto[n+1]{m+1}y^{n+1}(x^{m}) \lonlyif \exists z^n\ x^{m} = \uptypeto[n]{m}z^n)$
    \item $\forall x^n\forall y^0 \lnot \uptypeto[0]{n+1} y^0(x^n)$.
\end{listbullet}
To prove that \STTu and \CTT[\omega] are definitionally equivalent (Theorem \ref{thm:CTTomega:DE}), we first define an interpretation, $I$, from \CTT[\omega] to \STTu. This preserves the interpretation of all logical symbols, including $=$; its only non-trivial action is as follows:\footnote{So: $[x^n = y^n]^I \coloneq x^n = y^n$; $[\phi \land \psi]^I \coloneq (\phi^I \land \psi^I)$; $[\lnot \phi]^I \coloneq \lnot \phi^I$; and $[\forall x^n\phi]^I \coloneq \forall x^n \phi^I$.}
\begin{align*}
    [y^n(x^m)]^I &\coloneq y^n(\uptypeto[m]{n-1} x^m)
\end{align*}
Observe that if $n = m + 1$, then $[y^n(x^m)]^I$ is just $y^n(x^m)$. Here is a very simple fact about the relationship between $\uptype$ and the interpretations of $\inCTT$ and $\eqCTT$, which holds just by unpacking some definitions (the proof is left to the reader):
\begin{lem}\label{lem:sttu:upisneat}
    Where $i = \max(m, n)$, \STTu proves: 
    \begin{listn-0}
        \item\label{sttu:eqCTT} $[x^m \eqCTT y^n]^I \liff \uptypeto{i} x^m = \uptypeto{i}y^n$ 
        \item\label{sttu:inCTT} $[x^m \inCTT y^n]^I \liff \uptypeto{i+1}y^n(\uptypeto{i}x^m)$
    \end{listn-0}
\end{lem}\noindent
We now have one substantial result:
\begin{lem}\label{lem:AllRules:Admissible}
    The $I$-interpretations of \AllE{n}{m} and \AllI{n}{m} are admissible in \STTu
\end{lem}
\begin{proof}
    We start with \AllE{n}{m}. Suppose that both $\phi(x^n)$ and $\phi(a^m)$ are well-formed in \CTT[\omega]. Working in \STTu, suppose $\forall x^n\phi^I(x^n)$. If $m = n$, then $\phi^I(a^m)$ follows by ordinary $\forall$E in \STTu. So consider the case when $m < n$. The variable $x^n$ cannot occur in any identity-claim, e.g.\ $x^n = c^n$, since $a^m = c^n$ is ill-formed in \CTT[\omega]; so $\phi$ must have this kind of shape (illustratively):
        $$\psi(x^n(v^i), \ldots, y^{k}(x^n), \ldots)$$
    with $i < m < n < k$; note that $i < m$, since $\phi(a^m)$ is well-formed in \CTT[\omega]. Now $\phi^I$ is: 
        $$\psi(x^n(\uptypeto[i]{n-1}v^{i}), \ldots  y^k(\uptypeto[n]{k-1}x^n), \ldots)$$
    Using \STTu's rule $\forall$E$^n$, we can infer $\phi^I(\uptypeto[m]{n} a^m)$, i.e.:
        $$\psi(\uptypeto[m]{n} a^m(\uptypeto[i]{n-1}v^{i}), \ldots y^k(\uptypeto[n]{k-1}\uptypeto[m]{n} a^m), \ldots)$$
    Simplifying, and using generalized Up-Possess, we obtain:
        $$\psi(a^m(\uptypeto[i]{m-1}v^{i}), \ldots y^k(\uptypeto[m]{k-1} a^m), \ldots)$$
    which is precisely $\phi^I(a^m)$, as required. 
    
    The admissibility of \AllI{n}{m} under interpretation follows straightforwardly. Given $\phi^I(b^n)$, with $b^n$ suitably arbitrary: infer $\forall x^n\phi^I(x^n)$ using \STTu's rule $\forall$I$^n$; with $a^m$ suitably arbitrary, infer $\phi^I(a^m)$ using \AllI{n}{m} under interpretation; finally, infer $\forall x^m\phi^I(x^m)$ using \STTu's rule $\forall$I$^m$.  
\end{proof}\noindent
It is now easy to prove that $I$ is an interpretation:
\begin{lem}\label{lem:from-CTTomega-to-STTu}
    $I : \CTT[\omega] \functionto \STTu$ is an interpretation.
\end{lem}
\begin{proof}
    We simply check all inference rules and axioms. Lemma \ref{lem:AllRules:Admissible} deals with the quantifier-rules, and no Limit-rules apply since we are considering \CTT[\omega]. 
    
    \emph{\CTT[\omega]-Comprehension}. If $\phi(x^n)$ is an \CTT[\omega]-formula, then $\phi^I(x^n)$ is an \STTu-formula; now use \STTu-Comprehension.
    
    \emph{Type-Founded}. Suppose $[a^m \inCTT b^{n+1}]^I$ i.e.\ $\uptypeto{i+1}b^{n+1}(\uptypeto{i}a^m)$ with $i = \max(m, n+1)$ by Lemma \ref{lem:sttu:upisneat}.\ref{sttu:inCTT}. By generalized Up-Founded, there is $z^{n}$ such that $\uptypeto{i}a^m = \uptypeto{i} z^n$, i.e.\ $[a^m \eqCTT z^n]^I$ by Lemma \ref{lem:sttu:upisneat}.\ref{sttu:eqCTT}. 
    
    \emph{Type-Base}. By generalized Up-Base, $\lnot\uptypeto[0]{n+1} y^0(x^n)$; so $[x^n \notinCTT y^0]^I$ by Lemma \ref{lem:sttu:upisneat}.\ref{sttu:inCTT}.
\end{proof}\noindent
We now switch to working in \CTT[\omega]. It will help if we allow ourselves the use of a definite description operator, $\maththe$, within \CTT[\omega]. (This is harmless since, by standard Russellian techniques, this can always be eliminated from any formula.)\footnote{\label{fn:harmlessfudge}So we are relying on the fact that \CTT[\omega] augmented with this device is definitionally equivalent to \CTT[\omega]. Clearly, it is. Still, for details of how to handle function symbols more austerely, see e.g.\ \textcite[\S5.5, esp. fn.23]{ButtonWalsh:PMT}.} Now, by Type-Raising in \CTT[\omega], for any type $n$ and each $x^n$ there is a unique $x^{n+1}$ such that $x^n \eqCTT x^{n+1}$; we will denote this in \CTT[\omega] using $\upsurrogate x^n$, i.e.\ $\upsurrogate x^n \coloneq \TheAbs{x^{n+1}}{x^n \eqCTT x^{n+1}}$. As before, we write $\upsurrogateto[m]{n} a^m$ for the result of applying $n-m$ instances of $\upsurrogate$ to $a^m$, yielding a type $n$ entity. We now define an interpretation, $J$, from \STTu to \CTT[\omega], with these actions:\footnote{And: $[x^n = y^n]^J \coloneq x^n = y^n$; $[\phi \land \psi]^J \coloneq (\phi^J \land \psi^J)$; $[\lnot \phi]^J \coloneq \lnot \phi^J$; and $[\forall x^n\phi]^J \coloneq \forall x^n \phi^J$.}
\begin{align*}
    [y^{n+1}(x^n)]^J &\coloneq y^{n+1}(x^n)\\
    [\uptype x^n]^J &\coloneq \upsurrogate x^n
\end{align*}
\begin{lem}\label{lem:from-STTu-to-CTTomega}
    $J : \STTu \functionto \CTT[\omega]$ is an interpretation.
\end{lem}
\begin{proof}
    \CTT[\omega]-Comprehension immediately licenses \STTu-Comprehension. For Up-Inject, suppose $\upsurrogate x^n = \upsurrogate y^n$, i.e.\ $\TheAbs{x^{n+1}}{x^n \eqCTT x^{n+1}} = \TheAbs{y^{n+1}}{y^n \eqCTT y^{n+1}}$; so $x^n \eqCTT y^n$ by Lemma \ref{fact:CTT}, and hence $x^n = y^n$. Similarly, Up-Possess holds by Lemma \ref{fact:CTT}. And Up-Founded and Up-Base hold via Type-Founded and Type-Base.
\end{proof}\noindent
It only remains to show that $I$ and $J$ together yield a definitional equivalence. 
\begin{lem}\label{lem:CTTomega:DE}
    \CTT[\omega] proves this scheme: $[[y^n(x^m)]^I]^J \liff y^n(x^m)$.
\end{lem}
\begin{proof}
    Note that $[[y^n(x^m)]^I]^J\text{ iff }
        [y^n(\uptypeto[m]{n-1}x^m)]^J\text{ iff }
        y^n(\upsurrogateto[m]{n-1} x^m) \text{ iff }
        y^n(x^m)$, using Lemma \ref{fact:CTT} for the final biconditional.
\end{proof}
\begin{lem}\label{lem:STTu:DE}
    \STTu proves these schemes: 
        $[[y^{n+1}(x^n)]^J]^I \liff y^{n+1}(x^n)$
        and 
        $[[\uptype x^n]^J]^I = \uptype x^n$. 
\end{lem}
\begin{proof}
    The first scheme is trivial. For the second:
    \begin{align*}
        [[\uptype x^n]^J]^I = 
            [\upsurrogate x^n]^I
        & = \TheAbs{x^{n+1}}{x^n \eqCTT x^{n+1}}^I\\
        & = \TheAbs{x^{n+1}}{\forall z^{n+2}(z^{n+2}(x^n) \liff z^{n+2}(x^{n+1}))}^I\\
        & = \TheAbs{x^{n+1}}{\forall z^{n+2}(z^{n+2}(\uptype x^n) \liff z^{n+2}(x^{n+1}))}\\
        & = \TheAbs{x^{n+1}}{\uptype x^n  = x^{n+1}}\\
        &= \uptype x^n
    \end{align*}
\end{proof}\noindent
Assembling Lemmas \ref{lem:from-CTTomega-to-STTu}--\ref{lem:STTu:DE}, we obtain:
\begin{thm}\label{thm:CTTomega:DE}
    \STTu and \CTT[\omega] are definitionally equivalent
\end{thm}

\section{Definitional equivalence for \FJT}
\label{equivalence:FJT}
In \S\ref{type theories:CTTfj}, we stated that \FJT is definitionally equivalent to \STTd. In this appendix, we define \STTd and prove the equivalence. 

The guiding idea is to simulate \FJT by using a version of \STT with this sort of behaviour: for all types $1 < m < n$, each type $n$ entity $a^n$, projects downwards to some type $m$ entity $\mathord{\downarrow_m}a^n$; we can then simulate an application $a^n(x^{m-1})$ by instead considering $\mathord{\downarrow_m}a^n(x^{m-1})$. However, there is a small snag: we are treating   $\mathord{\downarrow_m}$ as \emph{functional}; but, if we assume no version of extensionality, then we will have no way to decide whether $a^2$ should project downwards to $b^1$ or $c^1$, if $b^1$ and $c^1$ are coextensional. The snag can be avoided by using a relational (rather than function) version of downward-projection. What follows spells this out rigorously.

We define \STTd by augmenting \STT as follows. For each $n > 0$, we have a relational constant, $\typedownrel$, expressing the downward-projecting relation from a type $n\mathord{+}1$ entity to a type $n$ entity. So we write e.g.\ $e^4 \typedownrel d^3$ or $b^1 \typedownrel a^0$;\footnote{As with the signs $=$, $\eqCTT$ or $\inCTT$, we are using the same symbol (in a typically ambiguous way) for each type level.} when convenient, we may write $d^3 \typeuprel e^4$ or $a^0 \typeuprel b^1$ instead. We introduce some useful abbreviations:
    \begin{align*}
        a^n \coextensive b^n & \mliffdf 
            \forall x^{n-1}(a^n(x^{n-1}) \liff b^n(x^{n-1}))\text{, when }n>0\\
        a^n \typedowneq b^n & \mliffdf 
            \forall x^{n-1}(a^n \typedownrel x^{n-1} \liff b^n \typedownrel x^{n-1})\text{, when }n >1\\
        a^1 \typedowneq b^1 & \mliffdf a^1 = a^1 
    \end{align*}
So $a^n \coextensive b^n$ tells us that $a^n$ and $b^n$ are coextensive, and $a^n \typedowneq b^n$ tells us that $a^n$ and $b^n$ project downwards to exactly the same entities. The special stipulation for $a^1 \typedowneq b^1$ is needed as \STTd has \emph{no} relational constant $\typedownrel$ expressing a relation from a type $1$ entity to a type $0$ entity, and so will hold vacuously. We concatenate chains of conjunctions; so we may write e.g.\ $a^2 \typeuprel d^3 \typedownrel b^2 \coextensive c^2$ in place of $(d^3 \typedownrel a^2 \land d^3 \typedownrel b^2 \land b^2 \coextensive c^2)$. \STTd retains the \STT-Comprehension scheme for type $1$ entities, i.e.\ $\exists z^1 \forall x^0(z^1(x^0) \liff \phi(x^0))$; but for each $n > 1$, it has an augmented scheme:\footnote{It follows that some models of (plain vanilla) \STT cannot be turned into models of \STTd just by assigning some meaning to ``$\typedownrel$''. \emph{Example}: it is consistent with \STT that there are exactly four type $2$ entities; whereas \STTd (and \FJT) prove that there are at least eight type $2$ entities.}
\begin{listclean}
    \item[\emph{\STTd-Comprehension}.] $\forall y^n (\exists z^{n+1} \typedownrel y^n) \forall x^n(z^{n+1}(x^n)\liff \phi(x^n))$, for any formula $\phi(x^n)$ not containing $z^{n+1}$. 
\end{listclean}
For each $n > 0$, \STTd also has these axioms:
\begin{listclean}
    \item[\emph{Down$_\exists$}.]  $\forall z^{n+1}\exists x^n\phantom{)}z^{n+1} \typedownrel x^n$ 

    \item[\emph{Down$_\text{Sim}$}.] 
    $\forall z^{n+1}\forall x^{n}\forall y^n(x^n \typeuprel z^{n+1} \typedownrel y^n \lonlyif x^n \coextensive y^n \typedowneq x^{n})$
    
    \item[\emph{Down$_\text{Max}$}.] 
    $\forall z^{n+1}\forall x^{n}\forall y^{n}(z^{n+1} \typedownrel x^n \coextensive y^n \typedowneq x^{n} \lonlyif z^{n+1} \typedownrel y^{n})$
\end{listclean}
So Down$_\exists$ says that all entities of types $\geq 2$ project downwards; Down$_\text{Sim}$ says that if $z^{n+1}$ projects to two entities $x^n$ and $y^n$, then $x^n$ and $y^n$ apply and project to exactly the same entities; and Down$_\text{Max}$ says that if $z^{n+1}$ projects to some entity $x^n$, then $z^{n+1}$ also projects to any $y^n$ which applies and projects to exactly the same entities as $x^n$. These axioms ensure that $\typedownrel$-chains are always equivalent, in a strong sense which is brought out by these next two lemmas: 
\begin{lem}\label{lem:typedowneq}
    \STTd proves this scheme (with $n > 0$). If $a^{n+1} \typedownrel x^n$ and $b^{n+1} \typedownrel x^n$ for some $x^n$, then $a^{n+1} \typedowneq b^{n+1}$.
\end{lem}
\begin{proof}
    Suppose $a^{n+1} \typedownrel x^n$ and $b^{n+1} \typedownrel x^n$. If $a^{n+1} \typedownrel y^n$, then $x^n \coextensive y^n \typedowneq x^n$ by Down$_\text{Sim}$, so $b^{n+1} \typedownrel y^n$ by Down$_\text{Max}$; similarly if $b^{n+1} \typedownrel y^n$ then $a^{n+1} \typedownrel y^n$.     
\end{proof}
\begin{lem}\label{lem:thechain}
    \STTd proves this scheme  (with $n > 0$). Given any $\typedownrel$-chains:
    \begin{align*}
    a^{n+1} \typedownrel {}&a^{n} \typedownrel a^{n-1} \typedownrel \ldots \typedownrel a^{1}\\
        &b^{n} \typedownrel b^{n-1} \typedownrel \ldots \typedownrel b^{1}
    \end{align*}
    \begin{listn-0}
        \item\label{lem:thechain:coext} If $a^{n+1} \typedownrel b^{n}$, then $\bigland_{1 \leq i \leq n} a^{i} \coextensive b^{i} \typedowneq a^{i}$.
        \item\label{lem:thechain:gag} If there is $m$ such that $1 \leq m \leq n$ and $b^m \typedowneq a^m$ and $\bigland_{m \leq i \leq n} a^{i} \coextensive b^{i}$, then $\bigland_{m \leq i \leq n} b^{i} \typedowneq a^{i}$ and $a^{n+1} \typedownrel b^{n}$.
    \end{listn-0}
\end{lem}
\begin{proof}    
    \emphref{lem:thechain:coext} From Down$_\text{Sim}$, by induction.
    
    \emphref{lem:thechain:gag} By assumption, $a^{m+1} \typedownrel a^m \coextensive b^m \typedowneq a^m$, so $a^{m+1} \typedownrel b^m$ by Down$_\text{Max}$; so $b^{m+1} \typedowneq a^{m+1}$ by Lemma \ref{lem:typedowneq}. This establishes a base case; the rest follows by induction. Now $a^{n+1} \typedownrel b^{n}$ by Down$_\text{Max}$.
\end{proof}\noindent
We will use these results to prove that \STTd and \FJT are definitionally equivalent (Theorem \ref{thm:CTTfj:DE}). We first define an interpretation, $I$, to take us from \FJT to \STTd:\footnote{We choose variables to avoid clashes; $I$'s other actions are trivial.}
\begin{align*}
    [y^n(x^m)]^I &\coloneq \forall y^{n-1}\forall y^{n-2}\ldots \forall y^{m+1}(y^n \typedownrel y^{n-1} \typedownrel y^{n-2} \typedownrel \ldots \typedownrel y^{m+1} \lonlyif y^{m+1}(x^m))
\end{align*}
Note that if $m + 1 = n$, then $[y^n(x^m)]^I$ is just $y^n(x^m)$.
\begin{lem}\label{lem:from-CTTfj-to-STTd}
    $I : \FJT \functionto \STTd$ is an interpretation
\end{lem}
\begin{proof}
    For all $0 \leq i < n$, let $\phi_i$ be \FJT-formulas not containing $z^n$ or $z^j$ or $y^j$ for any $0 \leq j < n$. (No generality is lost here, as we can relabel variables as necessary.) By multiple successive applications of \STTd-Comprehension, there are $z^1 \typeuprel z^2 \typeuprel \ldots \typeuprel z^n$ such that $\forall x^i(z^{i+1}(x^i) \liff \phi_i^I(x^i))$, for each $0 \leq i < n$. By Lemma \ref{lem:thechain}.\ref{lem:thechain:coext}, for each $0 \leq i < n$, we have:
    \begin{align*}
        \forall x^i(\forall y^{n-1}\ldots \forall y^{i+1}(z^{n} \typedownrel y^{n-1} \typedownrel \ldots \typedownrel y^{i+1} \lonlyif y^{i+1}(x^i)) &\liff \phi_i^I(x^i))\\
        \text{i.e.\ } \forall x^i([z^{n}(x^i)]^I &\liff \phi_i^I(x^i))
    \end{align*}
    Conjoining these biconditionals and applying $\exists$I, we obtain $[\exists z^{n}\bigland_{i < n}\forall x^i(z^n(x^i) \liff \phi_i(x^i))]^I$, i.e.\ an arbitrary instance of $[\text{\FJT-Comprehension}]^I$.
\end{proof}\noindent
We now switch to working in \FJT. We introduce another abbreviation, for a bounded version of coextensiveness, whenever $k \leq \min(m,n)$:  
    \begin{align*}
        a^m \coextensiveCTT{k} b^n \mliffdf \bigland_{i < k} \forall x^i(a^m(x^i) \liff b^n(x^i))
    \end{align*}
Note the bound is $i < k$. We now define an interpretation, $J$, from \STTd to \FJT:\footnote{We choose variables to avoid clashes in $[y^{n+1} \typedownrel x^{n}]^J$; $J$'s other actions are trivial.}
\begin{align*}
    [y^{n+1}(x^n)]^J &\coloneq y^{n+1}(x^n)\\
    [y^{n+1} \typedownrel x^n]^J &\coloneq y^{n+1} \coextensiveCTT{n} x^{n}
\end{align*}
\begin{lem}\label{lem:de:helper}
    \FJT proves the following schemes, where  well-formed: 
    \begin{listn-0}
        \item\label{de:FJT:1} If $a^{m} \coextensiveCTT{k} c^l$ and $b^{n} \coextensiveCTT{k} c^l$, then $a^m \coextensiveCTT{k} b^n$
        \item\label{de:FJT:low} $a^m \coextensiveCTT{k} b^n$ iff $\forall x^k(a^{m} \coextensiveCTT{k} x^k \liff b^{n} \coextensiveCTT{k} x^k)$
        \item\label{de:FJT:2} $a^n \coextensiveCTT{n-1} b^n$ iff $[a^n \typedowneq b^n]^J$, noting here that we must have $n > 1$.
    \end{listn-0}
\end{lem}
\begin{proof}
    \emphref{de:FJT:1} Trivial. 
    
    \emphref{de:FJT:low} \emph{Left-to-right.} Suppose $a^{m} \coextensiveCTT{k} b^{n}$; if $a^{m} \coextensiveCTT{k} c^k$, then $b^{n} \coextensiveCTT{k} c^k$ by \eqref{de:FJT:1}, and conversely. \emph{Right-to-left}. Suppose $\forall x^k(a^{m} \coextensiveCTT{k} x^k \liff b^{n} \coextensiveCTT{k} v^k)$; by \FJT-Comprehension, there is some $c^k \coextensiveCTT{k} a^{m}$; so $b^{n} \coextensiveCTT{k} c^k$, and now  $a^{n} \coextensiveCTT{k} b^{n}$ by \eqref{de:FJT:1}. 
    
    \emphref{de:FJT:2} Using \eqref{de:FJT:low}, since $[a^{n} \typedowneq b^{n}]^J$ is $\forall x^{n-1}(a^{n} \coextensiveCTT{n-1} x^{n-1} \liff b^{n} \coextensiveCTT{n} x^{n-1})$.
 \end{proof}
\begin{lem}\label{lem:from-STTd-to-CTTfj}
    $J : \STTd \functionto \FJT$ is an interpretation.
\end{lem}
\begin{proof}
    \emph{For Down$_\exists$.} Fix $z^{n+1}$; by \FJT-Comprehension there is some $x^n \coextensiveCTT{n} z^{n+1}$, i.e.\ $[z^{n+1} \typedownrel x^{n}]^J$. 
    
    \emph{For Down$_\text{Sim}$.} Suppose $[x^n \typeuprel z^{n+1} \typedownrel y^n]^J$, i.e.\ $x^n \coextensiveCTT{n} z^{n+1} \coextensiveCTT{n} x^n$; so $x^n \coextensiveCTT{n} y^n$ by Lemma \ref{lem:de:helper}.\ref{de:FJT:1}. In particular, $x^n \coextensive y^n$, so $[x^n \coextensive y^n]^J$. Moreover, if $n > 1$ then $x^n \coextensiveCTT{n-1} y^n$, so that $[x^n \typedowneq  y^n]^J$ by Lemma \ref{lem:de:helper}.\ref{de:FJT:2}; if $n = 1$ then $[x^n \typedowneq  y^n]^J$ vacuously.
    
    \emph{For Down$_\text{Max}$.} Suppose $[z^{n+1} \typedownrel x^{n} \coextensive y^{n} \typedowneq x^n]^J$, i.e.\ $z^{n+1} \coextensiveCTT{n} x^n \coextensive y^{n} \coextensiveCTT{n-1} x^{n}$, using Lemma \ref{lem:de:helper}.\ref{de:FJT:2}. So $y^{n} \coextensiveCTT{n} x^{n}$, and hence $z^{n+1} \coextensiveCTT{n} y^{n}$ by Lemma \ref{lem:de:helper}.\ref{de:FJT:1}, i.e.\ $[z^{n+1} \typedownrel y^{n}]^J$.    
    
    \emph{For \STTd-Comprehension.} Let $\phi$ be any \STTd-formula not containing $z^{n+1}$ (but which may contain $y^n$). Fix $y^n$; by \FJT-Comprehension, there is $z^{n+1}$ such that:
    \begin{align*}
        \forall x^n(z^{n+1}(x^n) \liff \phi^J(x^n)) &\land \bigland_{i < n}\forall x^i(z^{n+1}(x^i) \liff y^n(v^i))\\
        \text{i.e.\ }
        \forall x^n(z^{n+1}(x^n) \liff \phi^J(x^n)) &\land z^{n+1} \coextensiveCTT{n} y^n\\
        \text{i.e.\ }
        [\forall x^n(z^{n+1}(x^n) \liff \phi(x^n))^J &\land z^{n+1} \typedownrel y^n]^J
    \end{align*}
    So we have arbitrary instances of $[\STTd\text{-Comprehension}]^J$.
\end{proof}\noindent
It only remains to show that $I$ and $J$ characterise a definitional equivalence. 
\begin{lem}
    \FJT proves this scheme: $[[y^n(x^m)]^I]^J \liff y^n(x^m)$.
\end{lem}
\begin{proof}
    Note that the following are equivalent:
    \begin{listn-0}
        \item $[[y^n(x^m)]^I]^J$
        \item $[\forall y^{n-1}\ldots \forall y^{m+1}(y^n \typedownrel y^{n-1} \typedownrel \ldots \typedownrel y^{m+1} \lonlyif y^{m+1}(x^m))]^J$
        \item $\forall y^{n-1}\ldots \forall y^{m+1}(y^n \coextensiveCTT{n-1} y^{n-1} \coextensiveCTT{n-2} \ldots \coextensiveCTT{m+1} y^{m+1} \lonlyif y^{m+1}(x^m))$
        \item $y^n(x^m)$
    \end{listn-0}
    The last equivalence uses Lemma \ref{lem:de:helper}.\ref{de:FJT:1}, and repeated instances of  \FJT-Comprehension to provide a chain $y^n \coextensiveCTT{n-1} a^{n-1} \coextensiveCTT{n-2} \ldots \coextensiveCTT{m+1} a^{m+1}$.
\end{proof}

\begin{lem}\label{lem:STTd:DE}
    \STTd proves these schemes:     
    $[[y^{n+1}(x^n)]^J]^I \liff y^{n+1}(x^n)$ and $[[y^{n+1} \typedownrel x^n]^J]^I \liff y^{n+1} \typedownrel x^n$. 
\end{lem}
\begin{proof}
    The first scheme is trivial. For the second, note that the following are equivalent:
      \begin{listn-0}
        \item $[[y^{n+1} \typedownrel x^n]^J]^I$
        \item $[\bigland_{i < n}\forall v^i(y^{n+1}(v^i) \liff x^n(v^i))]^I$
        \item\label{n:STTd:mess} 
        $\bigland_{i < n}\forall v^i(\forall y^n \forall y^{n-1}\ldots \forall y^{i+1}(y^{n+1} \typedownrel y^{n} \typedownrel y^{n-1} \typedownrel \ldots \typedownrel y^{i+1} \lonlyif y^{i+1}(v^i)) \liff {}$\\\phantom{.}\hspace{6.5em}
        $\forall x^{n-1} \ldots \forall x^{i+1}(x^n \typedownrel x^{n-1} \typedownrel \ldots \typedownrel x^{i+1} \lonlyif x^{i+1}(v^i)))$
        \item\label{n:STTd:goal} $y^{n+1} \typedownrel x^n$
    \end{listn-0}
    For the last equivalence, first note that repeated use of Down$_\exists$ gives us chains:
    \begin{align*}
        y^{n+1} \typedownrel{} &a^n \typedownrel a^{n-1} \typedownrel \ldots \typedownrel a^{1}\\
        &x^{n} \typedownrel b^{n-1} \typedownrel \ldots \typedownrel b^{1}
    \end{align*}
    Using Lemma \ref{lem:thechain}.\ref{lem:thechain:coext} twice, \eqref{n:STTd:mess} is equivalent to:
    \begin{listn}
        \item[(\ref{n:STTd:mess}$'$)] $a^n \coextensive x^n \land a^{n-1} \coextensive b^{n-1} \land \ldots \land a^1 \coextensive b^1$ 
    \end{listn}
    Now Lemma \ref{lem:thechain}.\ref{lem:thechain:coext} yields \eqref{n:STTd:goal} $\Rightarrow$ (\ref{n:STTd:mess}$'$), and Lemma \ref{lem:thechain}.\ref{lem:thechain:gag} gives (\ref{n:STTd:mess}$'$) $\Rightarrow$ \eqref{n:STTd:goal}.
\end{proof}\noindent
Assembling Lemmas \ref{lem:from-CTTfj-to-STTd}--\ref{lem:STTd:DE}, we obtain:
\begin{thm}\label{thm:CTTfj:DE}
    \STTd and \FJT are definitionally equivalent
\end{thm}
\end{subappendices}

\section*{Acknowledgements}
Thanks to 
Neil Barton,
Salvatore Florio,
Peter Fritz,
Luca Incurvati,
Stephan Krämer,
Øystein Linnebo,
Nicholas Jones,
Agustín Rayo,
Thomas Schindler,
Lukas Skiba, and
an anonymous referee for \emph{Review of Symbolic Logic}.

\printbibliography

\end{document}